\documentclass[12pt]{article}
\usepackage{amssymb,amsthm,amsmath,latexsym}
\usepackage{url}
\usepackage{lscape}
\newtheorem{thm}{Theorem}[section]
\newtheorem{prop}[thm]{Proposition}
\newtheorem{lem}[thm]{Lemma}
\newtheorem{cor}[thm]{Corollary}
\theoremstyle{remark}
\newtheorem{rem}[thm]{Remark}
\newtheorem{ex}[thm]{Example}
\newcommand{\FF}{\mathbb{F}}
\newcommand{\ZZ}{\mathbb{Z}}
\newcommand{\RR}{\mathbb{R}}
\DeclareMathOperator{\wt}{wt}
\DeclareMathOperator{\srg}{srg}

\begin{document}

\title{Unbiased weighing matrices of weight $9$}

\author{
Makoto Araya\thanks{Department of Computer Science,
Shizuoka University,
Hamamatsu 432--8011, Japan.
email: \texttt{araya@inf.shizuoka.ac.jp}},
Masaaki Harada\thanks{
Research Center for Pure and Applied Mathematics,
Graduate School of Information Sciences,
Tohoku University, Sendai 980--8579, Japan.
email: \texttt{mharada@tohoku.ac.jp}},
Hadi Kharaghani\thanks{
Department of Mathematics and Computer Science, University of Lethbridge,
Lethbridge, Alberta, T1K 3M4, Canada.
email: \texttt{kharaghani@uleth.ca}}, 
\\
Sho Suda\thanks{
Department of Mathematics, National Defense Academy of Japan, Yokosuka, Kanagawa 239--8686,
Japan.
email: \texttt{ssuda@nda.ac.jp}}
  and 
Wei-Hsuan Yu\thanks{
Mathematics Department, National Central University, Taoyuan 32001, Taiwan. 
email: \texttt{whyu@math.ncu.edu.tw}. 
}
}

\maketitle

\begin{abstract}
We investigate unbiased weighing matrices of weight $9$ 
and provide a construction method using mutually suitable Latin squares. 
For $n \le 16$, we determine the maximum size  among sets of mutually
unbiased weighing matrices of order $n$ and weight $9$.
Notably, our findings reveal that $13$ is the
smallest order where such pairs exist, and $16$ is the first order for which
a maximum class of unbiased weighing matrices is found.
\end{abstract}

\section{Introduction}\label{Sec:1}

A \emph{weighing matrix} $W$ of order $n$ and weight $k$ is an 
$n \times n$ $(1,-1,0)$-matrix $W$ such that $W W^T=kI_n$, where $I_n$ is 
the identity matrix of order $n$ and $W^T$ denotes the transpose of $W$. 
A weighing matrix of order $n$ and weight $n$ is also called a \emph{Hadamard} matrix of order $n$. 
Two Hadamard matrices $H,K$ of order $n$ are said to be
\emph{unbiased} if $(1/\sqrt{n})HK^{T}$ is also
a Hadamard matrix of order $n$.
Hadamard matrices $H_1,H_2,\ldots,H_f$ are said to be 
\emph{mutually unbiased} if any distinct two are unbiased.
Determining the
maximum size among sets of mutually unbiased Hadamard matrices is a fundamental problem.
Much work has been done concerning this fundamental problem (see, e.g., \cite{BKR}, \cite{HK}, \cite{HKO}, \cite{K11}, \cite{KSS} and the references given therein).
The notion of unbiased Hadamard matrices was generalized in~\cite{BKR} and~\cite{HKO}
as follows:
Two weighing matrices $W_1,W_2$ of order $n$ and weight $k$ are said to be 
\emph{unbiased} if $(1/\sqrt{k})W_1 W_2^{T}$ is also 
a weighing matrix of order $n$ and weight $k$.
Weighing matrices $W_1,W_2,\ldots,W_f$ are said to be 
\emph{mutually unbiased} if any distinct two are unbiased.
Inspired by the many applications of mutually unbiased Hadamard matrices,
we study mutually unbiased weighing matrices.

If there exists a pair of unbiased weighing matrices of order $n$ and weight $k$,
then $k$ must be a perfect square~\cite{BKR}.
Many examples of unbiased weighing matrices of weight $4$ are given in~\cite{BKR}
for small orders, and the maximum number of mutually unbiased weighing matrices of weight $4$ is determined in~\cite[Theorem~3.5]{NS}. 
An example of $4$ mutually unbiased weighing matrices of order $16$ and
weight $9$ is given in~\cite{HKO}.
This motivates us to study further the existence and some applications of unbiased weighing matrices of weight $9$.
In what follows, we construct mutually unbiased weighing matrices
using mutually suitable Latin squares.
For $n \le 16$, we can determine the maximum size among sets of mutually
unbiased weighing matrices of order $n$ and weight $9$.

The article is organized as follows: 
In Section~\ref{Sec:2}, we provide definitions and some known results regarding weighing 
matrices, Latin squares, and ternary codes used in this article. 
In Section~\ref{Sec:Latin}, we describe a construction method for mutually unbiased 
weighing matrices using mutually suitable Latin squares. As a special case, we show that 
there exists a set of \( q - 1 \) mutually unbiased weighing matrices of order \( 4q \) and 
weight $9$ for any prime power \( q \) with \( q \ge 4 \). 
In Section~\ref{Sec:sdp}, we demonstrate the upper bounds of the maximum size among sets 
of mutually unbiased weighing matrices of order \( n \) and weight $9$. This is achieved by 
applying linear programming and 
semidefinite programming to some spherical three-distance sets obtained from mutually 
unbiased weighing matrices of order \( n \) and weight $9$, and comparing the two methods. 
In Section~\ref{Sec:code}, for a given weighing matrix \( W_1 \), we present a method to 
obtain essentially all weighing matrices \( W_2 \) such that \( W_1, W_2 \) are unbiased by 
considering the ternary code generated by the rows of \( W_1 \). Using this method, 
computer calculations reveal the existence of unbiased weighing matrices of order \( n \) and 
weight $9$ for \( n \le 24 \) in Section~\ref{Sec:small}. 
It is then concluded that $13$ is the smallest order for which unbiased weighing matrices of 
weight $9$ exist. Additionally, we determine the maximum size among sets of mutually 
unbiased weighing matrices of order \( n \) and weight $9$ for \( n \le 16 \). Notably, we 
obtain, for the first time, a set of $15$ mutually unbiased weighing matrices of order $16$ and 
weight $9$, achieving the second upper bound in~\cite[Corollary 10]{BKR} 
(see also Proposition~\ref{prop:f}). 
Finally, 
we present the state-of-the-art on the maximum size 
among sets of mutually unbiased weighing matrices of order \( n \) and weight $9$
for $ n \le 24$.
To ensure this article is both comprehensive and accessible for readers, 
the data pertaining to unbiased weighing matrices is given in Appendix.

\section{Preliminaries}\label{Sec:2}

This section provides definitions and some known results
of weighing matrices, Latin squares and ternary codes used in this article.
Throughout this article, $I_n$ denotes the identity matrix of order $n$ 
and $O$ denotes the zero matrix of appropriate size.

\subsection{Unbiased weighing matrices}
A \emph{weighing matrix} $W$ of order $n$ and weight $k$ is an 
$n \times n$ $(1,-1,0)$-matrix $W$ such that $W W^T=kI_n$.
Two weighing matrices $W_1,W_2$ of order $n$ and weight $k$ are said to be 
\emph{equivalent} if there exist $(1,-1,0)$-monomial matrices $P, Q$ with 
$W_1=P W_2 Q$. 
We denote two equivalent weighing matrices $W_1,W_2$ by $W_1 \cong W_2$.
A classification of weighing matrices of weight $k \le 5$ and 
weighing matrices of orders $n \le 11$ was done by Chan, Rodger and Seberry~\cite{CRS86}
(see also~\cite{HM} for weight $5$).
A classification of weighing matrices of orders $12, 13$
was done by Ohmori~\cite{O89}, \cite{O93}, respectively
(see also~\cite{HM} for order $12$).
A classification of weighing matrices of orders $14, 15, 17$ was done 
in~\cite{HM}.

Two weighing matrices $W_1,W_2$ of order $n$ and weight $k$ are said to be 
\emph{unbiased} if $(1/\sqrt{k})W_1 W_2^{T}$ is also 
a weighing matrix of order $n$ and weight $k$~\cite{HKO} (see also~\cite{BKR}).
Note that to prove that weighing matrices $W_1, W_2$ of weight $k$ are
unbiased, it is enough to show that $W_1W_2^T$ is a $(k,-k,0)$-matrix.
If a pair of unbiased weighing matrices of order $n$ and weight $k$ exists,
then $k$ must be a perfect square~\cite{BKR}.
Weighing matrices $W_1,W_2,\ldots,W_f$ are said to be 
\emph{mutually unbiased} if any distinct two are unbiased.
Determining the
maximum size among sets of mutually unbiased weighing matrices is a fundamental problem.
Many examples of unbiased weighing matrices of weight $4$ are given in~\cite{BKR}
for small orders.
An example of $4$ mutually unbiased weighing matrices of order $16$ and
weight $9$ is also given in~\cite{HKO}.

Let $W_1,W_2$ be unbiased weighing matrices of order $n$ and weight $k$.
If $P,Q$ are $n \times n$ $(1,-1,0)$-monomial matrices, then
$PW_1Q, PW_2Q$ are also unbiased weighing matrices.
Thus, for the determination of the maximum size among sets of mutually
unbiased weighing matrices of order $n$ and weight $k$, 
it is sufficient to consider mutually
unbiased weighing matrices for each weighing matrix of all inequivalent weighing matrices only.

Suppose that there exists a pair of unbiased weighing matrices $W_1, W_2$ of order $n$ and $k$.
Let $D$ be a $n \times n$ diagonal matrix whose diagonal entries are $1$ or $-1$.
Then $W_1, DW_2$ are unbiased.
Since there are many such diagonal matrices, there are many pairs of unbiased weighing matrices
of order $n$ and $k$.

\subsection{Latin squares}

A \emph{Latin square} of side $n$ on symbol set $\{1,2,\ldots,n\}$ is an $n \times n$
array in which 
each cell contains a single symbol from the symbol set such that each symbol occurs
exactly once in each row and exactly once in each column.
Two Latin squares $L_1, L_2$ of side $n$ on symbol set $\{1,2,\ldots,n\}$
are called \emph{orthogonal} if 
$L_1(a,b)=L_1(c,d)$ and 
$L_2(a,b)=L_2(c,d)$ implies $a=c$ and $b=d$, where $L_i(a,b)$ denotes
the $(a,b)$-entry of $L_i$ $(i=1,2)$.
Latin squares in which every distinct pair of Latin squares is orthogonal are called 
\emph{mutually orthogonal} Latin squares. 

The following proposition provides 
the maximum size among sets of mutually 
orthogonal Latin squares and the bounds. 
\begin{prop}[{\cite[Theorems~3.25, 3.28, 3.44]{JACD}}]\label{prop:latin}
Let $N(n)$ be 
the maximum size among sets of mutually orthogonal Latin squares of side $n$. 
Then, the following hold.
\begin{enumerate}
\item $N(n)\leq n-1$ for every integer $n \ge 2$. 
\item  If $n$ is a prime power, then $N(n)=n-1$. 
\item $N(12)\geq 5$. 
\end{enumerate}
\end{prop} 

For our construction, we need a special presentation of mutually 
orthogonal Latin squares as follows:
Two Latin squares $L_1, L_2$ of side $n$ on symbol set $\{1,2,\ldots,n\}$
are called \emph{suitable} if every superimposition
of each row of $L_1$ on each row of $L_2$ results in only one element of the form 
$(a, a)$ for some $a \in \{1,2,\ldots,n\}$.
Latin squares in which every distinct pair of Latin squares is suitable are called 
\emph{mutually suitable} Latin squares, see~\cite[Lemma~2.3]{HKO}.

The following construction method of mutually unbiased weighing matrices
using mutually suitable Latin squares
is known.

\begin{prop}[{\cite[Theorem~2.12]{HKO}}]\label{prop:HKO}
Assume there exists a weighing matrix of order $n$ and weight $k$  
and a set of $f$ mutually suitable Latin squares of side $n$. 
Then,
there exists a set of $f+1$ mutually unbiased weighing matrices of order $n^2$ and weight $k^2$.
\end{prop}

As an example, a set of $4$ mutually unbiased weighing matrices of order $16$ and
weight $9$ is obtained~\cite{HKO}.

\subsection{Ternary self-orthogonal codes}
Let $\FF_3=\{0,1,2\}$ denote the finite field of order $3$.
A \emph{ternary} $[n,k]$ \emph{code} $C$ is a $k$-dimensional vector subspace
of $\FF_3^n$.
The parameters $n, k$ are called the \emph{length}, \emph{dimension} for $C$,
respectively.
The \emph{weight} $\wt(x)$ of a vector $x$ of $\FF_3^n$ is 
the number of non-zero components of $x$.
A vector of $C$ is called a \emph{codeword}.
The minimum non-zero weight of all codewords in $C$ is called
the \emph{minimum weight} of $C$.
Two codes $C,C'$ are \emph{equivalent} if there exists a
monomial matrix $P$ over $\mathbb{F}_3$ with $C' = C \cdot P$,
where $C \cdot P = \{ x P\mid x \in C\}$.
%

The \emph{dual code} $C^{\perp}$ of a code
$C$ of length $n$ is defined as
$
C^{\perp}=
\{x \in \FF_3^n \mid x \cdot y = 0 \text{ for all } y \in C\},
$
where $x \cdot y$ is the standard inner product.
A code $C$ is \emph{self-orthogonal} if $C \subset C^\perp$, and
a code $C$ is \emph{self-dual} if $C=C^\perp$.
A self-orthogonal code $C$ is \emph{maximal} if $C$ is
the only self-orthogonal code containing $C$.
A self-dual code is automatically maximal.
A maximal self-orthogonal code of length $n$
has dimension $(n-1)/2$ if $n$ is odd,
$n/2$ if $n \equiv 0 \pmod 4$ and 
$n/2-1$ if $n \equiv 2 \pmod 4$ (see~\cite{MPS}).

In the context of weighing matrices, we consider the elements $0,1,2$ of $\FF_3$ as 
$0,1,-1 \in \ZZ$, respectively, and
in the context of ternary codes, we consider the elements $0,1,-1$ of $\ZZ$ as 
$0,1,2 \in \FF_3$, respectively, unless otherwise specified.
For a weighing matrix $W$,
we denote by $C_3(W)$ the ternary code generated by the rows of $W$ throughout this article.

\section{Unbiased weighing matrices and Latin squares}
\label{Sec:Latin}

In this section, we give a construction method of mutually unbiased weighing matrices
by using mutually suitable Latin squares.

\begin{lem}\label{lem:ah}
Let $W$ be a weighing matrix of order $n$ and weight $k$ with $i$-th column $w_i$. 
Define $C_i=w_i w_i^T$ for $i$ with $1\leq i\leq n$. 
Then, the following hold. 
\begin{enumerate}
\item $C_i$ is symmetric  for $i$ with $1\leq i\leq n$. 
\item $C_i C_j=O$ for $i,j$ with $1\leq i\neq j\leq n$.
\item $C_i^2=k C_i$ for $i$ with $1\leq i\leq n$. 
\item $\sum_{i=1}^n C_i=k I_n$. 
\end{enumerate} 
\end{lem}
\begin{proof}
The straightforward proof is omitted.
\end{proof}

\begin{lem}\label{lem:ls}
Let $W,C_i$ be the same as in Lemma~\ref{lem:ah}, and $L=(l(i,j))_{i,j=1}^t$ a Latin square of side $t$, where $t \ge n$. 
Set $C_i=O$ for $i$ with $n+1 \le i \le t$ if $t>n$. 
Then $\widetilde{W}=(C_{l(i,j)})_{i,j=1}^t$ is a weighing matrix of order $tn$ and weight $k^2$.
\end{lem}
\begin{proof}
The $(i,j)$-block of $\widetilde{W}\widetilde{W}^T$ is
\begin{align}\label{eq:ls}
\sum_{m=1}^t C_{l(i,m)}C^T_{l(j,m)}. 
\end{align}
When $i=j$, \eqref{eq:ls} is equal to $k^2 I_{n}$ by Lemma~\ref{lem:ah} (i), (iii), (iv). 
When $i\neq j$, \eqref{eq:ls} is equal to $O$ by Lemma~\ref{lem:ah} (i), (ii). 
Thus, $\widetilde{W}$ is a weighing matrix of order $tn$ and weight $k^2$. 
\end{proof}


\begin{thm}\label{thm:msls}
Assume there exists a weighing matrix of order $n$ and weight $k$, and 
a set of $f$ mutually suitable Latin squares of side $t$, where $t\geq n$.
Then
there exists a set of $f$ mutually unbiased weighing matrices of order $tn$ and weight $k^2$.
\end{thm}
\begin{proof}
Let $W$ be a weighing matrix of order $n$ and weight $k$, and
$\{L_1,L_2,\ldots,L_f\}$ a set of $f$ mutually suitable Latin squares
of side $t$.
Let $C_i$ be the same as in Lemma~\ref{lem:ah}. 
Let $m_1,m_2$ be distinct elements in $\{1,2,\ldots,f\}$. 
Let $l(i,j),l'(i,j)$ denote the $(i,j)$-entry of $L_{m_1},L_{m_2}$, respectively. 
Set 
\[
\widetilde{W}_{m_1}=(C_{l(i,j)})_{i,j=1}^t, \widetilde{W}_{m_2}=(C_{l'(i,j)})_{i,j=1}^t.
\] 
Since $t\geq n$,
by Lemma~\ref{lem:ls}, each $\widetilde{W}_{m_i}$ is a weighing matrix of order $tn$ and weight $k^2$
$(i=1,2)$. 

We claim that $\widetilde{W}_{m_1},\widetilde{W}_{m_2}$ are unbiased. 
We calculate the $(i,j)$-block of $\widetilde{W}_{m_1} \widetilde{W}_{m_2}^T$ as follows:
\begin{align}\label{eq:msls}
\text{The $(i,j)$-block of }\widetilde{W}_{m_1}\widetilde{W}_{m_2}^T
=\sum_{m=1}^t C_{l(i,m)}C^T_{l'(j,m)}.
\end{align}
Then there uniquely exists $a\in\{1,2,\ldots,t\}$ such that $l(i,a)=l'(j,a)$,  
and $l(i,m)\neq l'(j,m)$ for any $m\neq a$ since $L_{m_1},L_{m_2}$ are suitable. 
Then, by Lemma~\ref{lem:ah} (i)--(iii), \eqref{eq:msls} is equal to
\begin{align*}
C_{l(i,a)}C^T_{l'(j,a)}=C_{l(i,a)}^2=kC_{l(i,a)}.
\end{align*}
Since $C_{l(i,a)}$ is a $(1,-1,0)$-matrix,
$\frac{1}{k}\widetilde{W}_{m_1}{\widetilde{W}_{m_2}}^T$ is a $(1,-1,0)$-matrix. It
readily follows that $\frac{1}{k}\widetilde{W}_{m_1}{\widetilde{W}_{m_2}}^T$ is a
weighing matrix of weight $k^2$. Therefore,
$\widetilde{W}_{m_1},{\widetilde{W}_{m_2}}$ are unbiased.
\end{proof}

In order to illustrate the above theorem, we give an example.

\begin{ex}
Let $W=
\left(
\begin{array}{rrrrr}
    0 & 1 & 1 & 1 \\ 
    -1 & 0 & 1 & -1 \\  
    -1 & -1 & 0 & 1 \\ 
    -1 & 1 & -1 & 0 
\end{array}
\right)
$ be a weighing matrix of order $4$ and weight $3$. 
Let $L_1=\left(
\begin{array}{rrrrr}
 1 & 2 & 3 & 4 & 5 \\
 2 & 3 & 4 & 5 & 1 \\
 3 & 4 & 5 & 1 & 2 \\
 4 & 5 & 1 & 2 & 3 \\
 5 & 1 & 2 & 3 & 4 \\
\end{array}
\right), L_2=\left(
\begin{array}{rrrrr}
 1 & 3 & 5 & 2 & 4 \\
 2 & 4 & 1 & 3 & 5 \\
 3 & 5 & 2 & 4 & 1 \\
 4 & 1 & 3 & 5 & 2 \\
 5 & 2 & 4 & 1 & 3 \\
\end{array}
\right)$ be suitable Latin squares of side $5$. 
Then the matrices $C_i$ ($i\in\{1,2,\ldots,5\}$) are as follows: 
\begin{align*}
    C_1&=\left(
\begin{array}{rrrrr}
 0 & 0 & 0 & 0 \\
 0 & 1 & 1 & 1 \\
 0 & 1 & 1 & 1 \\
 0 & 1 & 1 & 1 \\
\end{array}
\right),
    C_2=\left(
\begin{array}{rrrrr}
 1 & 0 & -1 & 1 \\
 0 & 0 & 0 & 0 \\
 -1 & 0 & 1 & -1 \\
 1 & 0 & -1 & 1 \\
\end{array}
\right),
\\
    C_3&=\left(
\begin{array}{rrrrr}
 1 & 1 & 0 & -1 \\
 1 & 1 & 0 & -1 \\
 0 & 0 & 0 & 0 \\
 -1 & -1 & 0 & 1 \\
\end{array}
\right),
    C_4=\left(
\begin{array}{rrrrr}
 1 & -1 & 1 & 0 \\
 -1 & 1 & -1 & 0 \\
 1 & -1 & 1 & 0 \\
 0 & 0 & 0 & 0 \\
\end{array}
\right),
    C_5=O.
\end{align*}
Then the resulting unbiased weighing matrices $\widetilde{W}_1,\widetilde{W}_2$ are given as  follows: 
\begin{align*}
\widetilde{W}_1=\left(\begin{array}{ccccc}
 C_1 & C_2 & C_3 & C_4 & C_5 \\
 C_2 & C_3 & C_4 & C_5 & C_1 \\
 C_3 & C_4 & C_5 & C_1 & C_2 \\
 C_4 & C_5 & C_1 & C_2 & C_3 \\
 C_5 & C_1 & C_2 & C_3 & C_4 \\
\end{array}\right),
\widetilde{W}_2=\left(\begin{array}{ccccc}
C_1&C_3&C_5&C_2&C_4\\
C_2&C_4&C_1&C_3&C_5\\
C_3&C_5&C_2&C_4&C_1\\
C_4&C_1&C_3&C_5&C_2\\
C_5&C_2&C_4&C_1&C_3\\
\end{array}\right).
\end{align*}
\end{ex}

\begin{rem}
We consider the case $t=n$ of the above theorem.
We claim that one more weighing matrix can be added to the set of
$f$ mutually unbiased weighing matrices constructed in the above theorem.
Define a $(1,-1,0)$-matrix $W'$ to be $(w_j w_i^T)_{i,j=1}^n$, where
$w_i$ is the same as in Lemma~\ref{lem:ah}. 
Then $W'$ is a weighing matrix of weight $k$. Indeed, 
\begin{align*}
\text{the $(i,j)$-block of }W'W'^T&=\sum_{m=1}^n w_m w_i^T w_j w_m^T \\
&=\delta_{ij}k\sum_{m=1}^n w_m w_m^T \\ &=k^2 I_n, 
\end{align*}
where $\delta_{ij}$ denotes the Kronecker delta. 
Now we show that $W', \widetilde{W}_m$ are unbiased for any $m\in\{1,2,\ldots,f\}$, 
where $\widetilde{W_1},\widetilde{W_2},\ldots,\widetilde{W_f}$ are mutually unbiased weighing matrices constructed in the above theorem. 
Letting $l''(i,j)$ denote the $(i,j)$-entry of the Latin square $L_m$, there uniquely 
exists $a\in\{1,2,\ldots,n\}$ such that $l''(j,a)=i$.
Then
\begin{align*}
\text{the $(i,j)$-block of }W'\widetilde{W}_m^T
&=\sum_{m=1}^n w_m w_i^T w_{l''(j,m)} (w_{l''(j,m)})^T\\
&= k w_a w_{i}^T.
\end{align*}
Since $w_a w_{i}^T$ is a $(1,-1,0)$-matrix, $W', \widetilde{W}_{m}$ are unbiased. 
Thus, Proposition~\ref{prop:HKO} is obtained.  
This means that the above theorem is an extension of Proposition~\ref{prop:HKO}.
\end{rem}

An example of a weighing matrix of order $4$ and weight $3$ is 
\[
\left(
\begin{array}{rrrr}
  0 & 1 & 1 & 1 \\
  -1 & 0 & 1 & -1\\
  -1 & -1 & 0 & 1\\
  -1 & 1 & -1 & 0
\end{array}
\right).
\]
If there exists a set of $f$ mutually suitable Latin squares of side $t$, 
where $t\geq 4$, then 
there exists a set of $f$ mutually unbiased weighing matrices of order $4t$ and weight $9$
by Theorem~\ref{thm:msls}.
Since there exists a set of $5$ mutually orthogonal Latin squares of side $12$
(see Proposition~\ref{prop:latin} (iii)), equivalently $5$ mutually suitable Latin squares
by~\cite[Lemma~2.3]{HKO},
we obtain a set of $5$ mutually unbiased weighing matrices of order $48$ and weight $9$.

By Proposition~\ref{prop:latin} (ii) and~\cite[Lemma~2.3]{HKO},
there exists a set of $q-1$ mutually suitable Latin squares of side $q$ for any prime power $q$.
We then have the following corollary. 
\begin{cor}\label{cor::muwm}
Assume a weighing matrix of order $n$ and weight $n-1$ exists.
Let $q$ be a prime power with $q\geq n$. 
Then a set of $q-1$ mutually unbiased weighing matrices of order $nq$ and weight $(n-1)^2$ exists. 
\end{cor}

Since it is well known that there exists a weighing matrix of order $p+1$ and weight $p$ 
for an odd prime power $p$,
there exists a set of $q-1$ mutually unbiased weighing matrices of order $q(p+1)$ and 
weight $p^2$ for 
an odd prime power $p$ and a prime power $q$ with $q\geq p+1$.

\section{Upper bounds of the maximum size of unbiased weighing matrices}
\label{Sec:sdp}

The notation $W_{\max}(n,9)$ is used to denote
the maximum size  among sets of mutually
unbiased weighing matrices of order $n$ and weight $9$.
This section studies the upper bounds of $W_{\max}(n,9)$.
The results are then used in Section~\ref{Sec:small} 
for the determination of $W_{\max}(16,9)$ and 
giving bounds for $W_{\max}(n,9)$ for $n=17,18,\ldots,24$.

\subsection{Relationship with spherical three-distance sets}

In order to obtain upper bounds of $W_{\max}(n,9)$, we consider
some spherical three-distance sets.
For a positive integer $s$, a \emph{spherical $s$-distance set} is a collection of unit vectors 
in $\mathbb{R}^{n}$ such that
the set of Euclidean distances between any two distinct vectors has cardinality $s$. 
For a set $X\subset  \RR^{n}$,
define the following set:
\begin{equation}\label{eq:ax}
A(X)=\{\langle x,y\rangle \mid x,y\in X,x\neq y \}, 
\end{equation}
where $\langle x,y\rangle $ denotes
the standard inner product.
Let $S^{n-1}$ be the unit sphere in $\RR^n$. 
We first review the following lemma, which is a slightly different version 
of {\cite[Proposition~2.3]{NS}} 
and connects mutually unbiased weighing matrices and some spherical three-distance sets.

\begin{lem}[{\cite[Proposition~2.3]{NS}}]\label{lem:MUWM}
Let $f, n, k$ be positive integers such that $f\geq2$. 
The existences of the following sets are equivalent.
\begin{enumerate}
\item A set of $f$ mutually unbiased weighing matrices of order $n$ and weight $k$.  
\item A subset $X\subset  S^{n-1}$ with the property that 
$A(X)=\{\pm 1/\sqrt{k} ,0\}$
and there exists a partition $\{X_0,X_1,\ldots,X_f\}$ of $X$ such that each $X_i$ is an orthonormal basis.
\end{enumerate}
\end{lem}

According to Lemma~\ref{lem:MUWM}, upper bounds of $f$ are obtained from upper bounds of $|X|$.
Also, many examples of 
unbiased weighing matrices of order \( n \) and  weight $9$ are given in Section~\ref{Sec:small}
for \( n \le 24 \).
By Lemma~\ref{lem:MUWM},
this provides examples of  spherical three-distance sets $X \subset \RR^n$ 
with $A(X)=\{\pm 1/3 ,0\}$ satisfying the above condition.

\subsection{Upper  bounds based on linear programming}

Here, we consider a linear programming (LP) method to 
obtain upper bounds for spherical three-distance sets.
We now prepare to restate~\cite[Corollary 10]{BKR}.

The Gegenbauer polynomials $G_k^{n}(x)$ of degree $k$ with dimension parameter $n$
are defined using the following recurrence relation:
\begin{align*}
&G^{n}_{0}(x)=1, \\
&G^{n}_{1}(x)=x,\\
&G^{n}_{k}(x)={\frac{(2k+n-4) \, x \, G^{n}_{k-1}(x)-(k-1)G^{n}_{k-2}(x)}{k+n-3}} 
\text{ if }k\geq{2}.
\end{align*}
Delsarte's linear programming method is then formulated as follows:
\begin{lem}[{see also~\cite[Theorem~4.3]{DGS}}]\label{lem:LP}
Let $p_{\text{LP}}$ be the parameter of LP constraints. 
Let $X$ be a spherical $s$-distance set in $\RR^n$ with $A(X)=\{d_1,d_2,\ldots,d_s\}$. 
Then 
\[
|X|\leq \max
\left\{
1+\sum_{j=1}^s a_j
\, \middle|\,
\begin{array}{ll}
\sum_{i=1}^s a_i G^{n}_{k}(d_i)\geq -1\ ( k=1,2,\ldots, p_{\text{LP}}),\\
a_i\geq0 \ (i=1,2,\ldots,s)
\end{array}
\right\}.
\]
\end{lem}



Applying Lemma~\ref{lem:LP} to three-distance sets obtained from mutually unbiased weighing matrices by Lemma~\ref{lem:MUWM}, 
the following upper bounds of 
the maximum size  among sets of mutually unbiased weighing matrices 
are obtained.

\begin{prop}[{\cite[Corollary 10]{BKR}}]\label{prop:f}
Assume a set of $f$ mutually unbiased weighing matrices of order $n$ and weight $k$ exists.
Then
\[
f \le \frac{(n-1)(n+4)}{6}.
\]
If $3k -(n+2) >0$, then
\[
f\le \frac{k(n-1)}{3k-(n+2)}.
\]
\end{prop}

In Table~\ref{Tab:UB},
we list the upper bounds $Ub_{\text{LP}}(n)$ obtained by Proposition~\ref{prop:f}
for $n=10,11,\ldots,30$.

For orders $7, 8$,
the first two examples of sets of mutually unbiased weighing matrices of
weight $4$ achieve the second upper bound in the above proposition, 
as given in~\cite[Theorems~21, 22]{BKR}, respectively.
As a third example, in Section~\ref{Sec:small}
we present sets of $15$ mutually unbiased weighing matrices of
order $16$ and weight $9$, which also achieve the bound.

\subsection{Upper  bounds based on semidefinite programming}

Here, we consider a semidefinite programming (SDP for short) method to 
obtain other upper bounds for spherical three-distance sets. 
Semidefinite programming often outperforms linear programming.
Liu and Yu~\cite{Liu24} utilized an SDP method to derive upper bounds 
for spherical three-distance sets.
We now prepare to restate~\cite[Theorem~5]{Liu24}. 

Following~\cite{bachoc08}, we define the $(p_{\text{SDP}}-k+1)\times(p_{\text{SDP}}-k+1)$
matrices $Y_k^n(u,v,t)$ and $S_k^n(u,v,t)$, where $p_{\text{SDP}}$ is the parameter of SDP matrix constraints.
For $k=0,1,\ldots,p_{\text{SDP}}$,
define the following:
\begin{align*}
(Y_k^n(u,v,t))_{i,j=0}^{p_{\text{SDP}} - k} &= u^iv^j((1-u^2)(1-v^2))^{\frac{k}{2}}G_k^{n-1}\left(\frac{t-uv}{\sqrt{(1-u^2)(1-v^2)}}\right).
\end{align*}
Define $S_0^n(u,v,t)$ to be the  
$(p_{\text{SDP}}+1) \times (p_{\text{SDP}}+1)$ all-one matrix.
For $k=1,2,\ldots,p_{\text{SDP}}$,
define the following:
\begin{align*}
S_k^n(u,v,t) = 
\begin{cases}
\frac{1}{6} \sum\limits_{\sigma} \sigma Y_k^n(u,v,t) &\text{ if } (u,v,t) \ne (1,1,1),\\
O  &\text{ if } (u,v,t) = (1,1,1),
\end{cases}
\end{align*}
where $\sigma$ runs through the group of all permutations of the variables $u,v,t$ 
which acts on matrix coefficients in an obvious way.

\begin{prop}[{\cite[Theorem~5]{Liu24}}]
\label{prop:SDP-primal}
Let $p_{\text{LP}}, p_{\text{SDP}}$ be the parameters of LP constraints, SDP matrix constraints, respectively.
If $X \subset \RR^n$ is a spherical three-distance set with 
$A(X)=\{d_1,d_2,d_3\}$,
then $|X|$ is bounded above by the solution of the following 
SDP problem:
\begin{subequations}
\begin{align*}
& \underset{}{\text{maximize}}
& & 1+\frac{1}{3}(x_1+x_2+x_3)\\
& \text{subject to}
& & \begin{pmatrix}
1&0\\
0&0
\end{pmatrix}
+{\frac{1}{3}}
\begin{pmatrix}
0&1\\
1&1
\end{pmatrix}
(x_1+x_2+x_3) +
\begin{pmatrix}
0&0\\
0&1
\end{pmatrix}
\sum_{i=4}^{13} x_i 
\succeq 0, \\
&&& 3+x_1 G_{k}^{n}(d_1)+x_2 G_{k}^{n}(d_2)+x_3 G_{k}^{n}(d_3)\geq 0\ (k=1,2,\ldots,p_{\text{LP}}), \\
&&& 
\begin{aligned}
&S_{k}^{n}(1,1,1)+x_1 S_{k}^{n}(d_1,d_1,1)
+x_2 S_{k}^{n}(d_2,d_2,1)+x_3 S_{k}^{n}(d_3,d_3,1)\\
&+x_4 S_{k}^{n}(d_1,d_1,d_1)
+x_5 S_{k}^{n}(d_2,d_2,d_2)
+x_6 S_{k}^{n}(d_3,d_3,d_3)\\
&+x_7 S_{k}^{n}(d_1,d_1,d_2)
+x_8 S_{k}^{n}(d_1,d_1,d_3)
+x_9 S_{k}^{n}(d_2,d_2,d_1)\\
&+x_{10} S_{k}^{n}(d_2,d_2,d_3)
+x_{11} S_{k}^{n}(d_3,d_3,d_1)
+x_{12} S_{k}^{n}(d_3,d_3,d_2)\\
&+x_{13} S_{k}^{n}(d_1,d_2,d_3)  \succeq  0\ (k=0,1,\ldots,p_{\text{SDP}}), 
\end{aligned} \\
&&& x_j  \geq 0 \ (j=1,2,\ldots,13), 
\end{align*}
\end{subequations}
where the sign ``$\succeq 0$'' stands for ``is positive semidefinite''.
\end{prop}

\begin{rem}
In the above proposition, 
the variables $x_i$ represent the number of triple points in $X$ associated 
with certain combinations of $d_1, d_2, d_3,1$. For instance, $x_1$ is related to counting triple points in $X$, where the three inner product values are $(d_1, d_1, 1)$. In a similar setting, $x_2$ corresponds to $(d_2, d_2, 1)$, and so on up to $x_{13}$, which is associated with $(d_1, d_2, d_3)$.
\end{rem}


By Lemma~\ref{lem:MUWM},
if there exists a  set of $f$ mutually unbiased weighing matrices of order $n$ and weight $9$,
then there exists a spherical three-distance set $X \subset S^{n-1}$ with $A(X)=\{\pm 1/3,0\}$.
Applying Proposition~\ref{prop:SDP-primal}
for $A(X)=\{\pm 1/3,0\}$ and $n=10,11,\ldots,30$,
we have the upper bounds $Ub_{\text{SDP}}(n)$ 
of the maximum size  among sets of mutually unbiased weighing matrices of order $n$ and weight $9$
listed in Table~\ref{Tab:UB}.
The SDP problem was solved in MATLAB with \textsc{CVX toolbox}~\cite{DB} under the condition $p_{\text{LP}}= p_{\text{SDP}}=5$. 
Table~\ref{Tab:UB} shows that  an SDP method indeed provides a better upper bound than
an LP method and is used in our computation of Section~\ref{Sec:small}. 

\begin{table}[thb]
\caption{Upper bounds of $W_{\max}(n,9)$}
\label{Tab:UB}
\centering
\medskip
{\small
\begin{tabular}{c|c|c||c|c|c}
\noalign{\hrule height1pt}
$n$ &  $Ub_{\text{LP}}(n)$ & $Ub_{\text{SDP}}(n)$&
$n$ &  $Ub_{\text{LP}}(n)$ & $Ub_{\text{SDP}}(n)$\\
\hline
$10$ &   $5$ &  $5$ & $21$ &  $45$ & $45$ \\
$11$ &   $6$ &  $6$ & $22$ &  $63$ & $63$ \\
$12$ &   $7$ &  $7$ & $23$ &  $99$ & $99$ \\
$13$ &   $9$ &  $9$ & $24$ & $107$ & $96$ \\
$14$ &  $10$ & $10$ & $25$ & $116$ & $92$ \\
$15$ &  $12$ & $12$ & $26$ & $125$ & $90$ \\
$16$ &  $15$ & $15$ & $27$ & $134$ & $87$ \\
$17$ &  $18$ & $18$ & $28$ & $144$ & $85$ \\
$18$ &  $21$ & $21$ & $29$ & $154$ & $83$ \\
$19$ &  $27$ & $27$ & $30$ & $164$ & $81$ \\
$20$ &  $34$ & $34$ &&&\\
\noalign{\hrule height1pt}
\end{tabular}
}
\end{table}

\begin{rem}
We have that  $Ub_{\text{LP}}(n)=Ub_{\text{SDP}}(n)$ for $n=10,11,\ldots,23$.
This means that the integer parts of upper bounds obtained by
linear programming and semidefinite programming match.
However, the decimal parts do not meet.
For example, for $n=11$, we have that $6.42857$ and $6.4272$, respectively.
\end{rem}

\section{Unbiased weighing matrices and ternary codes}
\label{Sec:code}

In this section, for a given weighing matrix $W_1$
we give a method to
obtain essentially all weighing matrices $W_2$ such that $W_1,W_2$ are unbiased
by considering the ternary code $C_3(W_1)^\perp$.

\begin{prop}\label{prop:code}
Let $W_1, W_2$ be unbiased weighing matrices of order $n$ and weight $k$.
Assume $k \equiv 0 \pmod 3$.
Then, every row of $W_2$ is presented as a $C_3(W_1)^\perp$ codeword.
\end{prop}
\begin{proof}
Since $W_1,W_2$ are unbiased weighing matrices,
all entries of $W_1W_2^T$ are $0, \pm \sqrt{k}$.
It is trivial that $k$ is a square and $k \equiv 0 \pmod 3$ if and only if
$k=(3m)^2$ for some integer $m$.
Since $k=(3m)^2$ for some integer $m$,
we have that $W_1W_2^T \equiv O \pmod{3}$.
If we regard rows of $W_1,W_2$ as vectors of $\FF_3^n$,
then every row of $W_2$ is orthogonal to all rows of $W_1$.
This implies
that every row of $W_2$ is presented as a codeword of $C_3(W_1)^\perp$.
The result follows.
\end{proof}

\begin{rem}
\begin{enumerate}
\item It follows from the above proof that $C_3(W_2) \subset C_3(W_1)^\perp$.
\item Since $W_1W_1^T \equiv W_2W_2^T \equiv O \pmod 3$, 
$C_3(W_1), C_3(W_2)$ are self-orthogonal.
\item For the case $n=k$ (unbiased Hadamard matrices), 
a similar observation is given in~\cite{HK}.
\end{enumerate}
\end{rem}

Assume $k \equiv 0 \pmod 3$.
For a given weighing matrix $W_1$ of order $n$ and order $k$,
Proposition~\ref{prop:code} claims that
all weighing matrices $W_2$ such that $W_1,W_2$ are unbiased, 
are obtained as codewords of weight $k$ in $C_3(W_1)^\perp$ 
by changing rows and/or negating rows.


We describe how to find all weighing matrices  $W_2$ of order $n$ and weight $k$ 
such that $W_1,W_2$ are unbiased,
formed by codewords of weight $k$ in $C_3(W_1)^\perp$.
A classification method of weighing matrices based on a classification
of maximal self-orthogonal codes is given in~\cite{HM}.
Our method modifies that given in~\cite{HM}.
We define the following set:
\begin{multline*}
S_{k}(C_3(W_1)^\perp) 
=\{x=(x_1,x_2,\ldots,x_n) \in C_3(W_1)^\perp \mid \wt(x)=k, x_{i(x)}=1 \},
\end{multline*}
where
\[
i(x)=\min\{i \in \{1,2,\ldots,n\} \mid x_i \ne 0\}.
\]
By rescaling the rows, any weighing matrix equals a weighing matrix 
such that the first nonzero entry of each row is $1$.
This implies that considering only codewords of the form $x=(x_1,x_2,\dots,x_{n})$ 
with $x_{i(x)}=1$ suffices.
For a vector $x=(x_1,x_2,\ldots,x_{n})$ of $\FF_3^{n}$, we define the following:
\[
\overline{x}=(\overline{x_1},\overline{x_2},\ldots,\overline{x_{n}}) \in \ZZ^{n},
\]
where 
$\overline{0},\overline{1},\overline{2}$ are $0, 1,-1 \in \ZZ$, respectively.
Set $\overline{S_{k}(C_3(W_1)^\perp)}=\{
\overline{x} \mid x \in S_{k}(C_3(W_1)^\perp)\}$. 
A graph $\Gamma_{k}(C_3(W_1)^\perp)$ is constructed as follows:
\begin{enumerate}
\item The vertex set is 
\[
\{\overline{x}  \in \overline{S_{k}(C_3(W_1)^\perp)} \mid  
\langle \overline{x},r_i \rangle \in 
\{0,\pm \sqrt{k}\} \ (i=1,2,\ldots,n)\} \subset \ZZ^{n},
\]
where $r_i$ denotes the $i$-th row of $W_1$.
\item Two vertices $\overline{x}$ and $\overline{y}$ are adjacent
if $\langle \overline{x},\overline{y}\rangle =0$.
\end{enumerate}
Clearly, an $n$-clique in $\Gamma_{k}(C_3(W_1)^\perp)$ gives a weighing matrix 
$W_2$ of order $n$
and weight $k$ such that $W_1,W_2$ are unbiased.
All weighing matrices $W_2$
formed by codewords of $C_3(W_1)^\perp$ such that $W_1,W_2$ are unbiased,
are obtained
by finding all $n$-cliques in $\Gamma_{k}(C_3(W_1)^\perp)$.
By Proposition~\ref{prop:code}, 
one can obtain all weighing matrices $W_2$ such that $W_1,W_2$ are unbiased
by changing rows and/or negating rows.
In addition, by the following method,
we determine the existence of a set of $f$ mutually unbiased weighing matrices 
including a weighing matrix $W_1$ for $f \ge 3$.
Another graph $G(W_1)$ is constructed as follows:
\begin{enumerate}
\item The vertex set is 
$\{W_2 \mid W_1,W_2 \text{ are unbiased}\}$.
\item 
Two vertices $W_2,W_2'$ are adjacent if $W_2,W_2'$ are unbiased.
\end{enumerate}
Clearly, an $(f-1)$-clique in $G(W_1)$ gives the existence of 
a set of $f$ mutually unbiased weighing matrices including $W_1$.

By the above method, our computer calculation investigates the existence of 
 unbiased weighing matrices of order $n$ and weight $9$ for $n \le 24$
in Section~\ref{Sec:small}.
All computer calculations in the section
were performed using programs in \textsc{Magma}~\cite{Magma} and the language \textsc{C}
except for the computer calculations to find cliques
and to test the isomorphism of association schemes.
The computer calculations for finding cliques
were performed using {\sc Cliquer}~\cite{Cliquer}.
The computer calculations for testing the isomorphism of association schemes
were performed using {\sc AssociationSchemes}~\cite{GAPAS}.

\begin{rem}
The result in this section may be generalized to codes over the
finite field $\FF_p$ of order $p$, where $p$ is a prime with $p \ge 5$ as follows:
Let $W_1, W_2$ be unbiased weighing matrices of order $n$ and weight $k$.
Assume $k \equiv 0 \pmod p$.
Then, every row of $W_2$ is presented as a codeword of the dual code of
the code over $\FF_p$ generated by the rows of $W_1$.
In the study of unbiased weighing matrices of weight $25$,
codes over $\FF_5$ may be useful.
\end{rem}

\section{Unbiased weighing matrices of small orders}
\label{Sec:small}

In this section,
our computer calculations investigate the existence of 
unbiased weighing matrices of order $n$ and weight $9$ for $n \le 24$
by the method in the previous section.
We also determine $W_{\max}(n,9)$ $(n \le 16)$ and 
provide bounds for $W_{\max}(n,9)$ $(n=17,18,\ldots,24)$.

\subsection{Known classifications}

A classification of weighing matrices of order $n$ and weight $9$ is known
for $n \le 18$.
The numbers $N(n,9)$ of the inequivalent weighing matrices of order $n$ and weight $9$
are listed in Table~\ref{Tab:W} along with references.
The inequivalent weighing matrices can be obtained electronically from~\cite{Data}.
According to the order given in~\cite{Data}, we denote the 
$N(n,9)$ weighing matrices of order $n$ and weight $9$
by $W_{n,1}, W_{n,2}, \ldots,W_{n,N(n,9)}$ for $n \le 18$.

\begin{table}[thb]
\caption{Known classification of weighing matrices of order $n$ and weight $9$}
\label{Tab:W}
\centering
\medskip
{\small
\begin{tabular}{c|c|c}
\noalign{\hrule height1pt}
$n$ & $N(n,9)$ & References\\
\hline
$1,2,\ldots,9,11$ & $0$ & \cite{CRS86}\\
$10$ & $9$ & \cite{CRS86}\\
$12$ & $4$  & \cite{O89} \\
$13$  &$8$ & \cite{O93} \\
$14$ & $7$  & \cite{HM} \\
$15$ & $37$& \cite{HM} \\
$16$ & $704$ &\cite{HM} \\
$17$ & $2360$ & \cite{HM} \\
$18$  &$11891$ & \cite{HM} \\
\noalign{\hrule height1pt}
\end{tabular}
}
\end{table}

\subsection{Orders 10, 12, 14}

By the method in Section~\ref{Sec:code}, we verified that
there exists no weighing matrix $W'$ such that $W,W'$ are unbiased
in $\Gamma_9(C_3(W)^\perp)$ for 
\[
W=W_{10,1}, W_{12,1}, W_{12,2}, W_{12,3},W_{12,4},
W_{14,1}, W_{14,2}, \ldots, W_{14,7}.
\]
Hence, we have the following proposition.

\begin{prop}\label{prop:101214}
No weighing matrix of order $n$ and weight $9$ has an unbiased mate for $n=10, 12, 14$.
\end{prop}

\subsection{Order 13}


By the method in Section~\ref{Sec:code}, we verified that
there is a pair of unbiased weighing matrices in $\Gamma_9(C_3(W)^\perp)$
for $W=W_{13, 2}, W_{13, 5}$.
In addition, we verified that there is no pair of unbiased weighing matrices in $\Gamma_9(C_3(W)^\perp)$
for $W=W_{13, i}$ $(i=1,3,4,6,7,8)$.
This gives the following proposition.

\begin{prop}\label{prop:13}
A pair of unbiased weighing matrices of order $13$ and weight $9$ exists.
\end{prop}


If $n \le 9$ and $n=11$, then
there exists no weighing matrix for order $n$ and weight $9$ (see Table~\ref{Tab:W}).
From Proposition~\ref{prop:101214}, 
there exists no pair of unbiased weighing matrices of
order $n$ and weight $9$ for $n=10, 12$.
Thus, we have the following corollary.

\begin{cor}\label{cor:1}
The smallest order for which there exists a pair of unbiased weighing matrices of weight $9$
is $13$.
\end{cor}

In addition, we calculated that the maximum sizes among sets of mutually
unbiased weighing matrices of order $13$ and weight $9$ are $3$ for 
$W=W_{13,2}, W_{13,5}$.
This gives the following proposition.

\begin{prop}\label{prop:13-2}
The maximum size among sets of mutually
unbiased weighing matrices of order $13$ and weight $9$ is $3$.
\end{prop}

As an example,
the rows of $3$ mutually unbiased weighing matrices $W_{13,5}$, $A_{13,5,2}$, $A_{13,5,3}$ are listed 
in Table~\ref{Tab:13M}.

\subsection{Order 15}


By the method in Section~\ref{Sec:code}, we verified that
there is no pair of unbiased weighing matrices for $W_{15,i}$
$(i \in \Delta)$, where $\Delta=\{2,4,6,9,24\}$.
We also verified that there is a pair of unbiased weighing matrices for $W_{15,i}$ 
$(i \in \{1,2,\ldots,37\}\setminus \Delta)$.
In addition, we calculated the maximum size $N_m(W)$ among sets of mutually
unbiased weighing matrices of order $15$ and weight $9$ for $W=W_{15,i}$ 
$(i \in \{1,2,\ldots,37\}\setminus \Delta)$.
The maximum size $N_m(W)$ is listed in Table~\ref{Tab:15}.
From the table, we have the following proposition.

\begin{prop}\label{prop:15-2}
The maximum size among sets of mutually
unbiased weighing matrices of order $15$ and weight $9$ is $7$.
\end{prop}

\begin{table}[thb]
\caption{Unbiased weighing matrices of order $15$ and weight $9$}
\label{Tab:15}
\centering
\medskip
{\small
\begin{tabular}{c|c||c|c||c|c||c|c}
\noalign{\hrule height1pt}
$W$ & $N_m(W)$ & $W$ & $N_m(W)$ & $W$ & $N_m(W)$& $W$ & $N_m(W)$\\
\hline
$W_{ 15,  1}$& 2 & $W_{ 15, 13}$& 3 & $W_{ 15, 21}$& 2 & $W_{ 15, 30}$& 3 \\
$W_{ 15,  3}$& 2 & $W_{ 15, 14}$& 3 & $W_{ 15, 22}$& 2 & $W_{ 15, 31}$& 2 \\
$W_{ 15,  5}$& 2 & $W_{ 15, 15}$& 4 & $W_{ 15, 23}$& 3 & $W_{ 15, 32}$& 2 \\
$W_{ 15,  7}$& 3 & $W_{ 15, 16}$& 2 & $W_{ 15, 25}$& 3 & $W_{ 15, 33}$& 2 \\
$W_{ 15,  8}$& 2 & $W_{ 15, 17}$& 2 & $W_{ 15, 26}$& 7 & $W_{ 15, 34}$& 2 \\
$W_{ 15, 10}$& 2 & $W_{ 15, 18}$& 4 & $W_{ 15, 27}$& 2 & $W_{ 15, 35}$& 2 \\
$W_{ 15, 11}$& 3 & $W_{ 15, 19}$& 2 & $W_{ 15, 28}$& 2 & $W_{ 15, 36}$& 2 \\
$W_{ 15, 12}$& 7 & $W_{ 15, 20}$& 3 & $W_{ 15, 29}$& 3 & $W_{ 15, 37}$& 3 \\
\noalign{\hrule height1pt}
\end{tabular}
}
\end{table}

As an example, the rows of $7$ mutually unbiased weighing matrices 
$W_{15,12}$, $A_{15,12,2}$, $A_{15,12,3},\ldots,A_{15,12,7}$
are listed in Table~\ref{Tab:15-12}.

\subsection{Order 16}
\subsubsection{Determination of  $W_{\max}(16,9)$}

For $W_{16,46}$, for the first time,
we found weighing matrices 
$A_{16,46,2},A_{16,46,3},\ldots, A_{16,46,15}$
of order $16$ and weight $9$ 
such that $W_{16,46}$ and these matrices are $15$ mutually unbiased
by the method in Section~\ref{Sec:code}.
The rows of the matrices are listed in Tables~\ref{Tab:16} and~\ref{Tab:16-2}.

By Proposition~\ref{prop:f} (see also Table~\ref{Tab:UB}), 
we have that $W_{\max}(16,9) \le 15$.
Hence, we have the following proposition.

\begin{prop}\label{prop:16}
The maximum size among sets of mutually
unbiased weighing matrices of order $16$ and weight $9$ is $15$.
\end{prop}

Therefore, we have determined the maximum size  $W_{\max}(n,9)$ among sets of mutually
unbiased weighing matrices of order $n$ and weight $9$ for $n \le 16$.

\begin{rem}
As described above,
an example of $4$ mutually unbiased weighing matrices of order $16$ and
weight $9$ is given in~\cite{HKO}.
We improved the previously known maximum size among sets of mutually
unbiased weighing matrices of order $16$ and weight $9$.
\end{rem}

\begin{rem}
All inequivalent weighing matrices of order $16$ and weight $9$ are known (see Table~\ref{Tab:W}).
Due to the computational complexity,
finding all mutually
unbiased weighing matrices for each weighing matrix of order $16$ and weight $9$ 
seems infeasible. 
\end{rem}

\subsubsection{Some applications of  $W_{\max}(16,9)$}

For orders $7,8$,
the first two examples of sets of mutually unbiased weighing matrices of
weight $4$ achieve the second upper bound in Proposition~\ref{prop:f}, 
as given in~\cite[Theorems~21, 22]{BKR}, respectively.
As a third
example, we presented sets of $15$ mutually unbiased weighing matrices of
order $16$ and weight $9$, which also achieve the bound.
From the viewpoint of spherical designs, strongly regular graphs and association schemes
(see~\cite{S} for undefined terms), 
we investigate $15$ mutually unbiased weighing matrices of order $16$ and weight $9$. 

Let $\mathcal{W}=\{W_1,W_2,\ldots,W_{15}\}$
denote a set of $15$ mutually unbiased weighing matrices of order $16$ and weight $9$. 
Let $Y$ be the set of row vectors of 
$I_{16}$, $\frac{1}{3}W_{1},\frac{1}{3}W_{2},\ldots, \frac{1}{3}W_{15}$. 
Since the set $\mathcal{W}$
achieves the second upper bound in Proposition~\ref{prop:f}, 
equivalently the inequality~\cite[(3.9)]{CCKS}, by~\cite[Proposition~3.12]{CCKS}, 
the graph with vertex set $Y$ and edges determined by orthogonality is a strongly regular graph with parameters $(256,120,56,56)$.   

Set $X=Y\cup(-Y)$. By~\cite[Theorems~4.3 and~5.3]{DGS}, 
$X$ is a spherical $5$-design. 
Since $A(X)=\{\pm 1/3,0,-1\}$ (see~\eqref{eq:ax} for the definition of $A(X)$) and $X$ is antipodal, 
it follows from~\cite[Theorem~1.1]{BB09} that the pair $X$ and  the binary relations on $X$ determined by 
the elements $A(X)$ is a $4$-class bipartite $Q$-polynomial association scheme 
with the following second eigenmatrix: 
\[
Q_4=
\left(
\begin{array}{rrrrr}
 1 & 16 & 135 & 240 & 120 \\
 1 & \frac{16}{3} & 7 & -\frac{16}{3} & -8 \\
 1 & 0 & -9 & 0 & 8 \\
 1 & -\frac{16}{3} & 7 & \frac{16}{3} & -8 \\
 1 & -16 & 135 & -240 & 120 \\
\end{array}
\right). 
\]
Since the vertex set is decomposed into disjoint cross-polytopes, applying~\cite[Theorem~3.2]{S} to this association scheme, 
we have a $5$-class association scheme with the following second eigenmatrix:
\[
Q_5=
\left(
\begin{array}{rrrrrrr}
 1 & 16 & 135 & 240 & 105 & 15 \\
 1 & \frac{16}{3} & 7 & -\frac{16}{3} & -7 & -1 \\
 1 & 0 & -9 & 0 & 9 & -1 \\
 1 & -\frac{16}{3} & 7 & \frac{16}{3} & -7 & -1 \\
 1 & -16 & 135 & -240 & 105 & 15 \\
 1 & 0 & -9 & 0 & -7 & 15 \\
\end{array}
\right). 
\]
Conversely, as was shown in~\cite[Theorem~4.1]{S}, a set of $15$ mutually 
unbiased weighing matrices of order $16$ and weight $9$ is obtained from the $5$-class association scheme
with the above second eigenmatrix $Q_5$.


\subsubsection{More mutually unbiased  weighing matrices}
 
For $W_{16,i}$ $(i=562,569,695)$, we found weighing matrices 
$A_{16,i,2},A_{16,i,3},\ldots, A_{16,i,15}$
of order $16$ and weight $9$ 
such that $W_{16,i}$ and these matrices are $15$ mutually unbiased.
The rows of the matrices are listed in 
Tables~\ref{Tab:161-1}--\ref{Tab:163-2}.

We give some observations of the above $15$ mutually unbiased weighing matrices.
We verified that
\[
C_3(W_{16,i})=
C_3(A_{16,i,j})
\ (i=46,562,569,695, j=2,3,\ldots,15)
\]
and $C_3(W_{16,i})$ $(i=46,562,569,695)$ are  equivalent to the unique ternary self-dual code $2f_8$
and minimum weight $6$ given in~\cite[Section~VI]{CPS}.
Note that the code $2f_8$ is generated by the rows of 
$
\left(
\begin{array}{cc}
I_8 & H
\end{array}
\right)
$,
where $H$ is a Hadamard matrix of order $8$.

Let ${\mathcal W}_{16,i}$ $(i=1,2,3,4)$ denote 
the sets of $W_{16,j},A_{16,j,2},\ldots, A_{16,j,15}$ $(j=46,562,569,695)$, respectively.
For $i=1,2,3,4$,
let $\srg({\mathcal W}_{16,i})$ denote the strongly regular graph with parameters 
$(256,120,56,56)$ constructed from ${\mathcal W}_{16,i}$ via the above construction.
We verified the following:
\begin{enumerate}
\item
$\srg({\mathcal W}_{16,i})$ $(i=1,2,4)$ are isomorphic.
\item
$\srg({\mathcal W}_{16,1}), \srg({\mathcal W}_{16,3})$ are non-isomorphic.
\item
$\srg({\mathcal W}_{16,1}), \srg({\mathcal W}_{16,3})$
have automorphism groups of orders $43008, 21504$, respectively.
\end{enumerate}
For $d=4,5$ and $i=1,2,3,4$,
let $\mathcal{X}_d({\mathcal W}_{16,i})$ denote the $d$-class association scheme
with second eigenmatrix $Q_d$
constructed from ${\mathcal W}_{16,i}$ via the above construction.
We verified the following:
\begin{enumerate}
\item
$\mathcal{X}_4({\mathcal W}_{16,i})$ $(i=1,2,4)$ are isomorphic.
\item
$\mathcal{X}_4({\mathcal W}_{16,1}),\mathcal{X}_4({\mathcal W}_{16,3})$ are non-isomorphic.
\item
$\mathcal{X}_4({\mathcal W}_{16,1}), \mathcal{X}_4({\mathcal W}_{16,3})$
have automorphism groups of orders $86016, 43008$, respectively.
\item
$\mathcal{X}_5({\mathcal W}_{16,i})$ $(i=1,2,3,4)$ are non-isomorphic.
\item
$\mathcal{X}_5({\mathcal W}_{16,i})$ $(i=1,2,3,4)$
have automorphism groups of orders $24, 24, 672, 32$, respectively.
\end{enumerate}
This establishes that ${\mathcal W}_{16,i}$ $(i=1,2,3,4)$ are essentially different.

\subsection{Orders 17, 18}

%

For $W_{17,33}$ (resp.\ $W_{18,15}$),
by the method in Section~\ref{Sec:code},
we found weighing matrices 
$A_{17,33,2},A_{17,33,3},A_{17,33,4},A_{17,33,5}$ 
(resp.\ $A_{18,15,2},A_{18,15,3},A_{18,15,4}$)
of order $17$ (resp.\ $18$) and weight $9$ 
such that $W_{17,33}$ (resp.\ $W_{18,15}$) and these matrices are $5$ 
(resp.\ $4$) mutually unbiased.
The rows of the matrices are listed in Table~\ref{Tab:17} (resp.\ \ref{Tab:18}).

By Proposition~\ref{prop:f} (see also Table~\ref{Tab:UB}), 
we have that $W_{\max}(17,9) \le 18$, $W_{\max}(18,9) \le 21$.
Hence, we have the following proposition.

\begin{prop}\label{prop:1718}
Let $W_{\max}(n,9)$ denote 
the maximum size among sets of mutually
unbiased weighing matrices of order $n$ and weight $9$.
Then $W_{\max}(17,9) \in \{5,6,\ldots,18\}$, 
$W_{\max}(18,9) \in \{4,5,\ldots,21\}$.
\end{prop}

\begin{rem}
All inequivalent weighing matrices of order $n$ and weight $9$ are known
for $n=17,18$ (see Table~\ref{Tab:W}).
Due to the computational complexity,
finding all mutually
unbiased weighing matrices for each weighing matrix of weight $9$ for these orders seems infeasible. 
\end{rem}

\subsection{Order 19}

A classification of weighing matrices of order $19$ and weight $9$ has not yet been done.
By the method in~\cite{HM}, we found weighing matrices of
order $19$ and weight $9$ from ternary maximal self-orthogonal codes of length $19$.
Note that a classification of ternary maximal self-orthogonal codes of length $19$
was given in~\cite{PSW}.
Among the matrices found, there exists a weighing matrix $W_{19}$ 
such that 
$C_3(W_{19})^\perp$ contains $6$ mutually unbiased weighing matrices
by the method in Section~\ref{Sec:code}.
We denote the remaining $5$ matrices by
$A_{19,2},A_{19,3},A_{19,4},A_{19,5},A_{19,6}$.
The rows of the matrices are listed in Table~\ref{Tab:19}.
By Proposition~\ref{prop:f} (see also Table~\ref{Tab:UB}), 
we have that $W_{\max}(19,9) \le 27$.
Hence, we have the following proposition.

\begin{prop}\label{prop:19}
Let $W_{\max}(19,9)$ denote 
the maximum size among sets of mutually
unbiased weighing matrices of order $19$ and weight $9$.
Then
$W_{\max}(19,9) \in \{6,7,\ldots,27\}$.
\end{prop}

\subsection{Order 20}

By Corollary~\ref{cor::muwm} with $n=4$ and $q=5$, there exists a set of $4$ mutually unbiased weighing matrices 
of order $20$ and weight $9$.
A classification of weighing matrices of order $20$ and weight $9$ has not yet been done.
Similar to order $19$, we tried to find $5$ mutually unbiased weighing matrices 
of order $20$ and weight $9$.
However, an extensive search failed to discover such weighing matrices.
By Proposition~\ref{prop:f} (see also Table~\ref{Tab:UB}), 
we have that $W_{\max}(20,9) \le 34$.
Hence, we have the following proposition.

\begin{prop}\label{prop:20}
Let $W_{\max}(20,9)$ denote 
the maximum size among sets of mutually
unbiased weighing matrices of order $20$ and weight $9$.
Then $W_{\max}(20,9) \in \{4,5,\ldots,34\}$.
\end{prop}

\subsection{Orders 21, 22, 23, 24}

For $n=21,22,23,24$,
a classification of weighing matrices of order $n$ and weight $9$ has not yet been done.
By the method in~\cite{HM}, we found weighing matrices of
order $n$ and weight $9$ from ternary maximal self-orthogonal codes of lengths
$n=21,22,23$ and ternary self-dual codes of length $n=24$.
Note that a classification of ternary maximal self-orthogonal codes of lengths $21,22,23$
was given in~\cite{AHS} and a classification of ternary self-dual codes of length 
$24$ was given in~\cite{HM24}.
For $n=21,22,23,24$,
among the matrices found, there exists a weighing matrix $W_{n}$ of order $n$ and weight $9$
such that 
$C_3(W_{n})^\perp$ contains $N_n$ mutually unbiased weighing matrices
by the method in Section~\ref{Sec:code}, 
where $N_n=3,2,9,6$, respectively.
We denote the remaining $N_n-1$ matrices by
$A_{n,2},A_{n,3},\ldots,A_{n,N_n}$.
The rows of the matrices are listed in Tables~\ref{Tab:21}, \ref{Tab:23},
\ref{Tab:22}, \ref{Tab:22-2} and \ref{Tab:24}.
By Proposition~\ref{prop:f} (see also Table~\ref{Tab:UB}),
we have that $W_{\max}(21,9) \le 45$,
$W_{\max}(22,9) \le 63$,
$W_{\max}(23,9) \le 99$.
From $Ub_{\text{SDP}}(24)$ in Table~\ref{Tab:UB},
we have that $W_{\max}(24,9) \le 96$.
Hence, we have the following proposition.

\begin{prop}\label{prop:21222324}
Let $W_{\max}(n,9)$ denote 
the maximum size among sets of mutually
unbiased weighing matrices of order $n$ and weight $9$.
Then $W_{\max}(21,9) \in \{3,4,\ldots,45\}$,
$W_{\max}(22,9) \in \{2,3,\ldots,63\}$,
$W_{\max}(23,9) \in \{9,10,\ldots,99\}$,
$W_{\max}(24,9) \in \{6,7,\ldots,96\}$.
\end{prop}

For $n \ge 25$, 
a classification of ternary maximal self-orthogonal codes of length $n$ has not yet been done.
We stopped searching for mutually unbiased weighing matrices of
weight $9$ at order $24$.

\subsection{Maximum sizes of mutually unbiased weighing matrices}

Finally, we present in Table~\ref{Tab:max}  the state-of-the-art on the maximum size 
$W_{\max}(n,9)$ among sets of mutually unbiased weighing matrices of order \( n \) and weight $9$.

\begin{table}[thbp]
\caption{Mutually unbiased weighing matrices of order $n$ and weight $9$}
\label{Tab:max}
\centering
\medskip
{\small
\begin{tabular}{c|c|c}
\noalign{\hrule height1pt}
$n$ & $W_{\max}(n,9)$ & References\\
\hline
10 & $0$ & Proposition~\ref{prop:101214}\\
12 & $0$ & Proposition~\ref{prop:101214}\\
13 & $3$ & Proposition~\ref{prop:13-2}\\
14 & $0$ & Proposition~\ref{prop:101214}\\
15 & $7$ & Proposition~\ref{prop:15-2} \\
16 & $15$ & Proposition~\ref{prop:16} \\
17 & $5$--$18$ & Proposition~\ref{prop:1718}\\
18 & $4$--$21$ & Proposition~\ref{prop:1718}\\
19 & $6$--$27$ & Proposition~\ref{prop:19}\\
20 & $4$--$34$ & Proposition~\ref{prop:20}\\
21 & $3$--$45$ & Proposition~\ref{prop:21222324}\\
22 & $2$--$63$ & Proposition~\ref{prop:21222324}\\
23 & $9$--$99$ & Proposition~\ref{prop:21222324}\\
24 & $6$--$96$ & Proposition~\ref{prop:21222324}\\
\noalign{\hrule height1pt}
\end{tabular}
}
\end{table}

\bigskip
\noindent
\textbf{Acknowledgments.}
In this work, the ACCMS Kyoto University supercomputer was partially used.
Makoto Araya, Masaaki Harada and Sho Suda are
supported by JSPS KAKENHI Grant Numbers 21K03350, 23K25784, and
22K03410, respectively.
Hadi Kharaghani is supported by the Natural Sciences and
Engineering  Research Council of Canada (NSERC). 
Wei-Hsuan Yu is supported by  MOST  in  Taiwan under  Grant  MOST109-2628-M-008-002-MY4.


\bigskip
\noindent
\textbf{Data availability} 
All data generated and/or analyzed during the current study are available
from the authors upon request.

\bigskip
\noindent
\textbf{Code availability} Not applicable.

\bigskip
\noindent
\textbf{Declarations}

\noindent
\textbf{Conflict of interest} The authors declare there are no conflicts of interest.



\section*{Appendix}
Here we give mutually unbiased weighing matrices of order $n$ and weight $9$
for $n=13,15,16,\ldots,24$.
When weighing matrices are presented, 
$2$ denotes $-1$ to save space.

\begin{table}[thbp]
\caption{Mutually unbiased weighing matrices of order $13$ and weight $9$}
\label{Tab:13M}
\centering
\medskip
{\footnotesize
\begin{tabular}{l}
\noalign{\hrule height1pt}
\multicolumn{1}{c}{$W_{13,5}$}\\
\hline
0202212010112
0122221101010
1022102220011
1200022111021
2201100121110\\
2112110011001
1120110100122
2222001002121
1001211201101
0021010112211\\
1212011120200
1010121012110
0110202122101\\
\hline
\multicolumn{1}{c}{$ A_{13,5,2}$}\\
\hline
0122012012101
1001210022221
1111100101101
1120222001022
1010002222112\\
1022121122000
1102011201210
1212201010120
1221020210011
0110220112210\\
1200112110202
0012122200221
0101101212022\\
\hline
\multicolumn{1}{c}{$ A_{13,5,3}$}\\
\hline
1010110211102
1200101012221
0122100220222
0101221211200
1221010021210\\
0012121021011
0110012202211
1102210101021
1001022220121
1212222000202\\
1120122110010
1022201202110
1111001122002\\
\hline
\noalign{\hrule height1pt}
\end{tabular}
}
\end{table}

\begin{table}[thbp]
\caption{Mutually unbiased weighing matrices of order $15$ and weight $9$}
\label{Tab:15-12}
\centering
\medskip
{\footnotesize
\begin{tabular}{l}
\noalign{\hrule height1pt}
\multicolumn{1}{c}{$W_{15,12}$}\\
\hline
101122010212000
200012101211100
200021210011021
120200101110201\\
110100221001011
002202222210200
112201011200002
000000112221211\\
022010210002111
210222001022001
012020102110110
001202210101012\\
122022000021120
021221020200110
222120001000212\\
\hline
\multicolumn{1}{c}{$ A_{15,12,2}$}\\
\hline
001220210222002
012002211210001
110222020100210
110011201120020\\
101012020210022
120120220002101
102021002201220
100211100202011\\
120102111020200
012002102022120
010121101012002
122200010110102\\
001220101001121
111100012001110
001000012112221\\
\hline
\multicolumn{1}{c}{$ A_{15,12,3}$}\\
\hline
101011010020212
000102112201022
122200122001010
111220200211000\\
001102121010210
010210121220020
010212002112002
001002002222111\\
100012011101101
012111000210110
112020110002201
100110222000221\\
121001100012120
122002201202002
110120020120102\\
\hline
\multicolumn{1}{c}{$ A_{15,12,4}$}\\
\hline
120200022222020
010111012022001
101012102101200
012212020000111\\
010011220201202
101210211012000
000102221122200
120001200121110\\
112022212000002
001102200201121
110021020110021
102110101000122\\
111020101220010
122100000210211
001100022012112\\
\hline
\multicolumn{1}{c}{$ A_{15,12,5}$}\\
\hline
101201020021201
100120120211100
011112001021002
110001112100022\\
001111200110101
012220201101100
112010021012200
100210202200112\\
011202110010011
011022222002020
100000011222121
120002210011220\\
102102002120011
121020001102012
010121210200210\\
\hline
\multicolumn{1}{c}{$ A_{15,12,6}$}\\
\hline
102011022102100
102110210221000
100021011010111
101202221020100\\
012122020202010
120001121200202
010221002021012
111100200110202\\
010201200202221
011010110202102
001010022211011
010010101120211\\
112202100011020
101120102020021
120202012002210\\
\hline
\multicolumn{1}{c}{$ A_{15,12,7}$}\\
\hline
010122002011201
112001100021101
011000021212101
011221011100200\\
102021220010012
011010222120010
100212001201210
102202200102021\\
110100000222222
000122211020110
121100021101020
010010210011122\\
101222102000102
121001212200001
100110110112010\\
\hline
\noalign{\hrule height1pt}
\end{tabular}
}
\end{table}

\begin{table}[thbp]
\caption{Mutually unbiased weighing matrices of order $16$ and weight $9$}
\label{Tab:16}
\centering
\medskip
{\footnotesize
\begin{tabular}{l}
\noalign{\hrule height1pt}
\multicolumn{1}{c}{$W_{16,46}$}\\
\hline
2212111200000001
0200000022111212
1000000011111111
1221121200020000\\
0000200022122111
0000000221112122
1211212200200000
0020000022211121\\
0000002021121221
1111111120000000
1122211200000200
1212221100000020\\
2111221200001000
1112122202000000
0000020021212211
0001000012112221\\
\hline
\multicolumn{1}{c}{$ A_{16,46,2}$}\\
\hline
0010202212111000
0010201121011010
0010201200202221
0010102201200111\\
0111002101102020
0112021010020102
0122202001201020
1101010010200212\\
1101020022021001
1201020000210122
1102010022012100
1212002020020202\\
1002120110010201
1200210110020101
1020220201102010
1000111201101020\\
\hline
\multicolumn{1}{c}{$ A_{16,46,3}$}\\
\hline
0001111011020011
0001211022201200
0001212001110012
0001121020111100\\
0011220112120000
0012201221120000
0112100100001212
1011100200002222\\
1202011012110000
1110012000001121
1120000120102201
1120000202020112\\
1100221011210000
1210000120202110
1220000101021022
1200022200001211\\
\hline
\multicolumn{1}{c}{$ A_{16,46,4}$}\\
\hline
0100112101200201
0110201102200101
0101002202201022
0101001220022210\\
0101001111101002
0102101210012020
0102002102122002
1011022010012010\\
1021110002100101
1122020020011010
1012010210021010
1200101102200202\\
1010210020110220
1010120021020120
1020210001202102
1020220010020221\\
\hline
\multicolumn{1}{c}{$ A_{16,46,5}$}\\
\hline
0110011010011201
0110021022102020
0120022012100011
0120021000221202\\
0001222201010220
0012022120201001
0102012220101002
1111200001020110\\
1001100222210010
1001012102020220
1002100201022201
1002200212200120\\
1002200100012212
1120100101010120
1210120010101002
1220201020101001\\
\hline
\multicolumn{1}{c}{$ A_{16,46,6}$}\\
\hline
0010121101012002
0110222020100210
0011020202222100
0011112001021002\\
0012010120022021
0012010202210202
0012020211001021
1101010201012001\\
1211000102011001
1102002110200110
1202100220100110
1100201020001122\\
1100101012120200
1200202010102022
1200201001220210
1020122020200220\\
\hline
\multicolumn{1}{c}{$ A_{16,46,7}$}\\
\hline
0110012022220010
0110022001011102
0120012000012222
0120011021100110\\
0001122220102001
0011011202010120
0101021102010210
1001100110022102\\
1001200121200021
1001200200121202
1112200010102001
1002021220202002\\
1002100122101020
1120100210201001
1210110001010210
1220202002010110\\
\hline
\multicolumn{1}{c}{$ A_{16,46,8}$}\\
\hline
0100122110021020
0120201110022010
0101101201200102
0101002120010111\\
0102001120210220
0102001202101101
0102002211010210
1011021001100201\\
1021210010011020
1222020001200101
1012010101100102
1100202220022020\\
1010110012202001
1010220002201012
1020110020021210
1020120002112002\\
\noalign{\hrule height1pt}
\end{tabular}
}
\end{table}

\begin{table}[thbp]
\caption{Mutually unbiased weighing matrices of order $16$ and weight $9$}
\label{Tab:16-2}
\centering
\medskip
{\footnotesize
\begin{tabular}{l}
\noalign{\hrule height1pt}
\multicolumn{1}{c}{$ A_{16,46,9}$}\\
\hline
0110210211000022
0120022200212100
0111100012001110
0121200001001211\\
0121100020120022
0112100021110001
0122010112000021
1201102011000021\\
1001201100112100
1022101011000012
1002222000121100
1110020100222200\\
1000021202011220
1000011220220101
1000012202102210
1000012120211002\\
\hline
\multicolumn{1}{c}{$ A_{16,46,10}$}\\
\hline
0110110200122100
0120021221000021
0111200020212002
0121010100221100\\
0112200002021220
0122200010110102
0122100002202210
1101102000111200\\
1001201212000011
1012101000211100
1002212021000011
1220010222000022\\
1000021120120012
1000022102002121
1000022211220002
1000011111002220\\
\hline
\multicolumn{1}{c}{$ A_{16,46,11}$}\\
\hline
0010121210200210
0120111002011002
0011020120001222
0011010111210100\\
0011010220101011
0012112010100220
0012020102110110
1101020210100120\\
1221000110100210
1102001101021001
1202200201011002
1100202002210201\\
1100102021002012
1200102002221100
1200101020012021
1010211002022002\\
\hline
\multicolumn{1}{c}{$ A_{16,46,12}$}\\
\hline
0100110222000221
0100210121020202
0100120122200102
0100220100111021\\
0101212012002100
0121120011002200
0102220200220011
1201220022002200\\
1011001010221020
1021002221001100
1012002001212020
1022001000122120\\
1022002012001202
1110001200110012
1010102100120011
1020011100210011\\
\hline
\multicolumn{1}{c}{$ A_{16,46,13}$}\\
\hline
0010202100220122
0010102122012200
0010101110100121
0010101222021002\\
0111002210020201
0112011001202010
0122102020020101
1121001002202020\\
1201010021122000
1001120101201010
1102020001100222
1202010000201221\\
1202020012022010
1100210102101010
1010220220010101
1000112210010102\\
\hline
\multicolumn{1}{c}{$ A_{16,46,14}$}\\
\hline
0001112020202120
0001122002020212
0001221001022022
0001222010201101\\
0011210100002211
0121200222110000
0012202200002112
1201011000001112\\
1022200122220000
1110000221201200
1110000112010022
1120012011120000\\
1100121000002111
1210000202120021
1220000210212200
1200022121110000\\
\hline
\multicolumn{1}{c}{$ A_{16,46,15}$}\\
\hline
0100110110112010
0100210200211110
0100120201121010
0100220212002202\\
0101211000120021
0111120000210021
0102220121002100
1011002022100102\\
1021002000222011
1021001021010202
1102110000220022
1012001022000211\\
1022002200110021
1210001211002100
1010202111001200
1020021112001100\\
\noalign{\hrule height1pt}
\end{tabular}
}
\end{table}

\begin{table}[thbp]
\caption{Mutually unbiased weighing matrices of order $16$ and weight $9$}
\label{Tab:161-1}
\centering
\medskip
{\footnotesize
\begin{tabular}{l}
\noalign{\hrule height1pt}
\multicolumn{1}{c}{$W_{16,562}$}\\
\hline
1102100100020111
1001010111202002
1001010220120210
1201020112101000\\
0120202010010211
0220212000020121
0210201001201011
2101001100020221\\
2001020210022110
0021122021201000
2001010122010110
0120201001101102\\
1111000200010121
1020221022202000
0220101001112001
0010222121102000\\
\hline
\multicolumn{1}{c}{$ A_{16,562,2}$}\\
\hline
0100122110021020
0120201110022010
0111020022002102
0111010010211010\\
0121010021000222
0102101210012020
0112010022020201
1011021001100201\\
1021110002100101
1122020020011010
1012010101100102
1200101102200202\\
1000102201222010
1000201200221120
1000202100212021
1000202212100202\\
\hline
\multicolumn{1}{c}{$ A_{16,562,3}$}\\
\hline
0100012202210220
0100012120022012
0100021120201022
0100022211001012\\
0110011201120020
0120120220002101
0112020110002201
1101201002110010\\
1011100012220100
1021200001220201
1001102101110020
1012202020001101\\
1022200010002122
1022100002121200
1002111001210010
1210020220002202\\
\hline
\multicolumn{1}{c}{$ A_{16,562,4}$}\\
\hline
0001111011020011
0001112020202120
0001211022201200
0001122002020212\\
0111202011000220
0121021012000120
0122002200221001
1001220221000110\\
1112001000222002
1120000120102201
1100110200111002
1210000202120021\\
1220000210212200
1220000101021022
1010120100211001
1000212112000110\\
\hline
\multicolumn{1}{c}{$ A_{16,562,5}$}\\
\hline
0011120210100022
0101200200202212
0121100102011200
0012102110200012\\
0102100201100211
0102200212011100
0102200100120222
1101012001021100\\
1022100220200022
1100021101012100
1220011010100012
1200222001011200\\
1010012022112000
1010011010200221
1010021022021010
1020022012022001\\
\hline
\multicolumn{1}{c}{$ A_{16,562,6}$}\\
\hline
0010210121220020
0010110110012101
0010110222200012
0010120201021101\\
0101210212020001
0102222000202110
0122110011020002
1101120000202220\\
1201002000122102
1021001100201110
1102001022100120
1202001000022211\\
1202002012201020
1110002100101210
1010201211010002
1020012221010001\\
\hline
\multicolumn{1}{c}{$ A_{16,562,7}$}\\
\hline
0100011111110200
0100011220001111
0100021202122200
0100022102110101\\
0110012210002102
0120220201210020
0111020101220010
1201201010002101\\
1011200020101022
1021100020012012
1001102210001201
1012100021002021\\
1012200002210210
1022202001120010
1002121010001102
1120010102220020\\
\hline
\multicolumn{1}{c}{$ A_{16,562,8}$}\\
\hline
0001121102002011
0011001111001120
0011002202101110
0011002120210202\\
0012002102022120
0012021220210100
0102011102001012
1001011202002022\\
1212100010110200
1110210001002011
1120102010210100
1100220002001221\\
1100120021120002
1220200120110100
1200220010220012
1200110020221001\\
\noalign{\hrule height1pt}
\end{tabular}
}
\end{table}

\begin{table}[thbp]
\caption{Mutually unbiased weighing matrices of order $16$ and weight $9$}
\label{Tab:161-2}
\centering
\medskip
{\footnotesize
\begin{tabular}{l}
\noalign{\hrule height1pt}
\multicolumn{1}{c}{$ A_{16,562,9}$}\\
\hline
0001212001110012
0001121020111100
0001221001022022
0001222010201101\\
0111102000122001
0121011000212001
0122002121000120
1221002022000220\\
1002220200112001
1110000221201200
1110000112010022
1120000202020112\\
1210000120202110
1200110211000120
1020120111000210
1000211100121001\\
\hline
\multicolumn{1}{c}{$ A_{16,562,10}$}\\
\hline
0010210200111201
0010220212202020
0010220100011112
0010120122100220\\
0101210100102120
0102221021020001
0112110000201120
1101002022200011\\
1101001010021202
1201001021210020
1102002001012202
1202210022020002\\
1022001112010001
1210002111020001
1010101200102110
1020022200101120\\
\hline
\multicolumn{1}{c}{$ A_{16,562,11}$}\\
\hline
0001221110210200
0011001220222001
0011012102001021
0101022220110200\\
0012001120110011
0012001202001222
0012002211210001
1211100001001012\\
1002022120220200
1110220010120100
1120101001001021
1100210020210102\\
1100110012002210
1220200202001011
1200120002012120
1200210001102220\\
\hline
\multicolumn{1}{c}{$ A_{16,562,12}$}\\
\hline
0010222200020211
0110101012101001
0110102021010110
0110201020012220\\
0120101020220210
0120212022101000
0011211000020112
1121000211102000\\
1201010100010211
1001020120020121
1001020202211002
1102200111201000\\
1202001221101000
1002010202202101
1002020102102012
1010112000020222\\
\hline
\multicolumn{1}{c}{$ A_{16,562,13}$}\\
\hline
0000000112221211
0000001012212112
0000100011211222
0000010012121122\\
0010000011122212
0100000011222121
0001000012112221
1111111120000000\\
1121222100000002
1211212200200000
1221121200020000
1112122202000000\\
1122211200000200
1212221100000020
1222112100002000
1000000011111111\\
\hline
\multicolumn{1}{c}{$ A_{16,562,14}$}\\
\hline
0011110201012200
0101100222021020
0101100110100112
0101200121011001\\
0012101101021200
0102100122212000
0112200220100012
1011200102022200\\
1102012010100021
1100021210200011
1210011002011100
1200122020100011\\
1010022001200122
1020011021022100
1020012000201212
1020021000110222\\
\hline
\multicolumn{1}{c}{$ A_{16,562,15}$}\\
\hline
0100112101200201
0110201102200101
0101101201200102
0121020000121011\\
0112020001111020
0122010000112110
0122020012200202
1011022010012010\\
1021210010011020
1222020001200101
1012010210021010
1100202220022020\\
1000101110122020
1000101222010201
1000102122001102
1000201121000212\\
\noalign{\hrule height1pt}
\end{tabular}
}
\end{table}

\begin{table}[thbp]
\caption{Mutually unbiased weighing matrices of order $16$ and weight $9$}
\label{Tab:162-1}
\centering
\medskip
{\footnotesize
\begin{tabular}{l}
\noalign{\hrule height1pt}
\multicolumn{1}{c}{$W_{16,569}$}\\
\hline
1001010111202002
0020110121101010
0120202010010211
1110000221201200\\
2122000201022002
2201010220012200
1202001000022211
2010012000220111\\
0012112010100220
2101001100020221
0220122001200021
0001122002020212\\
0201200211121000
1020212022020020
0010222121102000
1101100200102101\\
\hline
\multicolumn{1}{c}{$ A_{16,569,2}$}\\
\hline
0000000112221211
0000001012212112
0000100011211222
0000010012121122\\
0010000011122212
0100000011222121
0001000012112221
1111111120000000\\
1121222100000002
1211212200200000
1221121200020000
1112122202000000\\
1122211200000200
1212221100000020
1222112100002000
1000000011111111\\
\hline
\multicolumn{1}{c}{$ A_{16,569,3}$}\\
\hline
0010222200020211
0010220100011112
0110101012101001
0120122020011200\\
0001211022201200
0011112001021002
0102200100120222
1101200211010020\\
1221000110100210
1001020120020121
1102010101200011
1202001221101000\\
1202002012201020
1120000202020112
1010012022112000
1010121000202202\\
\hline
\multicolumn{1}{c}{$ A_{16,569,4}$}\\
\hline
0001112020202120
0011002202101110
0101120210200201
0101200121011001\\
0121210010122000
0102222000202110
0112010022020201
1201220022002200\\
1022002200110021
1110000112010022
1100120021120002
1210002111020001\\
1010011201012010
1020012000201212
1020101102000111
1000201200221120\\
\hline
\multicolumn{1}{c}{$ A_{16,569,5}$}\\
\hline
0001111011020011
0101120101012020
0101200200202212
0111210002000121\\
0012002211210001
0102221021020001
0122010000112110
1021002221001100\\
1102110000220022
1110002100101210
1210000120202110
1200210001102220\\
1010021210100102
1020022012022001
1020201110211000
1000101222010201\\
\hline
\multicolumn{1}{c}{$ A_{16,569,6}$}\\
\hline
0100212221002001
0100022102110101
0120201001101102
0120102111020200\\
0101100222021020
0111010010211010
0102101210012020
1221010010000121\\
1011021001100201
1021100020012012
1012110000120110
1022020200201011\\
1002010120011220
1100201102222000
1010022001200122
1000202212100202\\
\hline
\multicolumn{1}{c}{$ A_{16,569,7}$}\\
\hline
0010210200111201
0110102021010110
0011001111001120
0121011000212001\\
0121020000121011
0112001210220010
0122202001201020
1111002000022220\\
1001020202211002
1102001022100120
1002120110010201
1200220201002101\\
1200110020221001
1020110211100002
1000201121000212
1000212112000110\\
\hline
\multicolumn{1}{c}{$ A_{16,569,8}$}\\
\hline
0010120201021101
0100021120201022
0011220112120000
0111100200110202\\
0012001120110011
0012102110200012
0102200212011100
1101002022200011\\
1201011000001112
1012200220022002
1022100002121200
1120021011002010\\
1100012101100120
1200120002012120
1200222001011200
1010011010200221\\
\noalign{\hrule height1pt}
\end{tabular}
}
\end{table}

\begin{table}[thbp]
\caption{Mutually unbiased weighing matrices of order $16$ and weight $9$}
\label{Tab:162-2}
\centering
\medskip
{\footnotesize
\begin{tabular}{l}
\noalign{\hrule height1pt}
\multicolumn{1}{c}{$ A_{16,569,9}$}\\
\hline
0010120122100220
0110202002221002
0011001220222001
0121010021000222\\
0121021012000120
0112001101002102
0122102020020101
1201002000122102\\
1001120101201010
1222001022200002
1002020211020220
1100220220110010\\
1100110012002210
1010110200011120
1000202100212021
1000211100121001\\
\hline
\multicolumn{1}{c}{$ A_{16,569,10}$}\\
\hline
0010110110012101
0100012202210220
0110011201120020
0110220111200200\\
0001212001110012
0111020022002102
0121010100221100
1201001021210020\\
1001201212000011
1001122010101020
1102101000011012
1012202020001101\\
1022200010002122
1120000120102201
1210010102020202
1000102201222010\\
\hline
\multicolumn{1}{c}{$ A_{16,569,11}$}\\
\hline
0010110222200012
0100021202122200
0011210100002211
0101100110100112\\
0012002102022120
0012101101021200
0122200122001010
1011200020101022\\
1021100102200220
1102002001012202
1202011012110000
1110021000210101\\
1100012210021001
1200122020100011
1200220010220012
1020011021022100\\
\hline
\multicolumn{1}{c}{$ A_{16,569,12}$}\\
\hline
0100011111110200
0120101020220210
0001112102120001
0001121020111100\\
0011002120210202
0012011102201100
0102012220101002
1111200001020110\\
1001011202002022
1012100021002021
1002020102102012
1120102010210100\\
1100220002001221
1210100210001210
1220210020010011
1220000101021022\\
\hline
\multicolumn{1}{c}{$ A_{16,569,13}$}\\
\hline
0100012120022012
0100222200120120
0110102100201021
0110201020012220\\
0101101201200102
0102100201100211
0122020012200202
1121010002111000\\
1011022010012010
1021200001220201
1012210011001002
1022020121010100\\
1002010202202101
1200102220000222
1010021022021010
1000101110122020\\
\hline
\multicolumn{1}{c}{$ A_{16,569,14}$}\\
\hline
0010210121220020
0100011220001111
0110012210002102
0110120100022011\\
0001222010201101
0112020001111020
0122010112000021
1101001010021202\\
1001201100112100
1001112001010201
1202101011200100
1012200002210210\\
1022202001120010
1120020220202020
1210000202120021
1000102122001102\\
\hline
\multicolumn{1}{c}{$ A_{16,569,15}$}\\
\hline
0100022211001012
0120102002102022
0001212110001220
0001221001022022\\
0011022220010021
0101021102010210
0012001202001222
1011100012220100\\
1001010220120210
1112200010102001
1002022120220200
1120101001001021\\
1100210020210102
1210100101110002
1220220002101100
1220000210212200\\
\noalign{\hrule height1pt}
\end{tabular}
}
\end{table}

\begin{table}[thbp]
\caption{Mutually unbiased weighing matrices of order $16$ and weight $9$}
\label{Tab:163-1}
\centering
\medskip
{\footnotesize
\begin{tabular}{l}
\noalign{\hrule height1pt}
\multicolumn{1}{c}{$W_{16,695}$}\\
\hline
0022122000010221
1020221022202000
1121000211102000
2101001100020221\\
0102012002022210
1010100112002120
1010100221220001
0101012020210022\\
1010200202101220
2201100222102000
1201010100010211
0220212000020121\\
0010222121102000
0102011020110101
1020100120121002
0101022002001111\\
\hline
\multicolumn{1}{c}{$ A_{16,695,2}$}\\
\hline
0001101111120100
0001101220011011
0001201201222200
0001202101210101\\
0011102210012002
0111200102100210
0012110101200210
1021100102200220\\
1102021010012001
1110020021001022
1120010020002112
1100012210021001\\
1210010020102021
1210020002220110
1220022001100210
1200211010011002\\
\hline
\multicolumn{1}{c}{$ A_{16,695,3}$}\\
\hline
0110110200122100
0120022200212100
0111000110011102
0121000121100011\\
0121000200021222
0112000201011201
0122010112000021
1201102011000021\\
1001201212000011
1022101011000012
1002222000121100
1110020100222200\\
1000121022110020
1000112022201010
1000211021200120
1000212000112202\\
\hline
\multicolumn{1}{c}{$ A_{16,695,4}$}\\
\hline
0010012120222100
0010011111010021
0010022211201100
0010022102010222\\
0100221112020100
0120102111020200
0102202200102021
1101200020201201\\
1201200001120022
1011120000102011
1021012012010100
1112010000101012\\
1202100002021021
1202200010202210
1022020121010100
1100101200202022\\
\hline
\multicolumn{1}{c}{$ A_{16,695,5}$}\\
\hline
0110101012101001
0110202002221002
0120102002102022
0120101020220210\\
0111012001002201
0121201010210020
0112001101002102
1211001002002202\\
1001110201001102
1001020120020121
1002010120011220
1002010202202101\\
1002020211020220
1100220220110010
1020210110120010
1000122110210010\\
\hline
\multicolumn{1}{c}{$ A_{16,695,6}$}\\
\hline
0010222200020211
0120212022101000
0101011011001210
0101021020122002\\
0012211012102000
0102021002201220
0102022011110002
1111000200010121\\
1201020112101000
1102100100020111
1202001221101000
1010200120210012\\
1010112000020222
1020200101002221
1020200210220102
1020100202012210\\
\hline
\multicolumn{1}{c}{$ A_{16,695,7}$}\\
\hline
0000000112221211
0000001012212112
0000100011211222
0000010012121122\\
0010000011122212
0100000011222121
0001000012112221
1111111120000000\\
1121222100000002
1211212200200000
1221121200020000
1112122202000000\\
1122211200000200
1212221100000020
1222112100002000
1000000011111111\\
\hline
\multicolumn{1}{c}{$ A_{16,695,8}$}\\
\hline
0001102202220120
0001102120002212
0001201122001022
0001202210101012\\
0011101201100220
0012120110021002
0122100220011002
1101021001200110\\
1012200220022002
1110010012210200
1120020002120201
1100012101100120\\
1210022020011001
1220010001221200
1220020010012022
1200111002100110\\
\noalign{\hrule height1pt}
\end{tabular}
}
\end{table}

\begin{table}[thbp]
\caption{Mutually unbiased weighing matrices of order $16$ and weight $9$}
\label{Tab:163-2}
\centering
\medskip
{\footnotesize
\begin{tabular}{l}
\noalign{\hrule height1pt}
\multicolumn{1}{c}{$ A_{16,695,9}$}\\
\hline
0010012202110011
0010011220201202
0010021120101110
0010021202022021\\
0100211100202011
0110102100201021
0101202221020100
1101100012020012\\
1201100020212100
1011220011010200
1021011000101021
1102200002010122\\
1102100021102200
1212010011020100
1022020200201011
1200202122020200\\
\hline
\multicolumn{1}{c}{$ A_{16,695,10}$}\\
\hline
0100210121020202
0100120201121010
0100120122200102
0100220100111021\\
0001112102120001
0101011202110002
0012012220001120
1221200020001110\\
1011001010221020
1001022220002220
1012001022000211
1022001000122120\\
1022002012001202
1110202010002110
1120110001210001
1210100101110002\\
\hline
\multicolumn{1}{c}{$ A_{16,695,11}$}\\
\hline
0110102021010110
0110201020012220
0120202010010211
0120201001101102\\
0111022010120020
0121101002002101
0112001210220010
1001010220120210\\
1001010111202002
1001020202211002
1122002020220020
1002110210110020\\
1002020102102012
1200220201002101
1010210102001101
1000121101001201\\
\hline
\multicolumn{1}{c}{$ A_{16,695,12}$}\\
\hline
0011220000202122
0101120210200201
0012210000220221
0012220012011010\\
0012120001100121
0102112010200102
0122220001022010
1011002101021010\\
1202120020200202
1100001222021100
1100001110100222
1100002122012001\\
1220001110200101
1200002212122000
1010011201012010
1020202201011020\\
\hline
\multicolumn{1}{c}{$ A_{16,695,13}$}\\
\hline
0100110110112010
0100110222000221
0100210200211110
0100220212002202\\
0001212110001220
0011021102210001
0102022120001210
1011002022100102\\
1021002000222011
1021001021010202
1222200002110001
1012002001212020\\
1002011102220002
1110201001120001
1120120010001120
1210100210001210\\
\hline
\multicolumn{1}{c}{$ A_{16,695,14}$}\\
\hline
0110210211000022
0120021221000021
0111000222202010
0121010100221100\\
0112000122120020
0122000212100110
0122000100212202
1101102000111200\\
1001201100112100
1012101000211100
1002212021000011
1220010222000022\\
1000122001022102
1000221000021212
1000222012200021
1000111010022201\\
\hline
\multicolumn{1}{c}{$ A_{16,695,15}$}\\
\hline
0011120022021200
0011210021111000
0011110010200111
0101120101012020\\
0012110022012002
0102111001021020
0112220020200101
1101210002022020\\
1012002110100201
1120001102011010
1100002201200212
1200002100201122\\
1200001121022010
1200001200210221
1010021210100102
1020102220100101\\
\noalign{\hrule height1pt}
\end{tabular}
}
\end{table}

\begin{table}[thbp]
\caption{Mutually unbiased weighing matrices of order $17$ and weight $9$}
\label{Tab:17}
\centering
\medskip
{\footnotesize
\begin{tabular}{l}
\noalign{\hrule height1pt}
\multicolumn{1}{c}{$W_{17,33}$}\\
\hline
02020202002110022
02021120000102011
10012020200012021
11201012020002100\\
01220120010001120
02202001101002102
10011220101101000
12000102100220220\\
00211101002010202
01100001102122020
10100000112210110
10220211010000201\\
00121001021210020
12000001222021100
10100110211100002
00002110120111001\\
11022020020000212\\
\hline
\multicolumn{1}{c}{$ A_{17,33,2}$}\\
\hline
11020000011200222
10222022000102100
10102021021001001
11201000102011001\\
01100010020110122
01220011220020010
10101201010102010
12001012001021100\\
00121120202200100
10010100202120220
01102112010000011
00001002221012201\\
10010100120202012
01010200000222121
12002210202210000
01010222200001012\\
00120202122020200\\
\hline
\multicolumn{1}{c}{$ A_{17,33,3}$}\\
\hline
01220110001210001
12212110000001002
10100011021022100
00122011012110000\\
11020002002021110
10100202201011020
01011210100010012
01010102010102120\\
01012000220100211
01100100022200222
00001010100121221
12000002122012001\\
10011001212200001
11202221100000020
00001121020111100
10100120111000210\\
10221000200102202\\
\hline
\multicolumn{1}{c}{$ A_{17,33,4}$}\\
\hline
10102201210020020
01222210000012002
01101010002222010
00000111002211120\\
00120000221001112
12001002210012100
01010020111200102
01010212000101101\\
11200100210001210
10011010021000222
10010221022010010
11020022020200021\\
00121200110011200
10221001100120100
10102102102100002
01100101001112001\\
12002010101200011\\
\hline
\multicolumn{1}{c}{$ A_{17,33,5}$}\\
\hline
01220000120102201
10220102002210002
10101110002100101
01010110110020202\\
00001120011202021
12000201102222000
10012002121000120
11022001011000110\\
01101020120201010
01010011220212000
00121212001002002
11201200200021020\\
10010202010010211
00122010000201221
10100021000110222
12000100221020210\\
01102022202022000\\
\noalign{\hrule height1pt}
\end{tabular}
}
\end{table}

\begin{table}[thb]
\caption{Mutually unbiased weighing matrices of order $18$ and weight $9$}
\label{Tab:18}
\centering
\medskip
{\footnotesize
\begin{tabular}{l}
\noalign{\hrule height1pt}
\multicolumn{1}{c}{$W_{18,15}$}\\
\hline
000112010220210002
210000221201000220
020021000022202220
220120202200100010\\
100220010210011200
102001201020110100
220000011111012000
101111000210002001\\
012000210012022002
211001100002110002
001120011000021120
000100220112211000\\
010022002000012121
012121100001200010
101022021000002012
200200001202200111\\
011000210120000211
001201202001200102\\
\hline
\multicolumn{1}{c}{$ A_{18,15,2}$}\\
\hline
001002001021020111
012011001220201000
010100010020102122
001012000010201122\\
100102000202222200
012100220111020000
010100100012010111
011010122120000200\\
102200102000020121
101001201102000021
100010111011100200
120011020000002112\\
000010212002121010
122102000120011000
000010212001212001
110202220200110000\\
101121002201001000
110220010100200012\\
\hline
\multicolumn{1}{c}{$ A_{18,15,3}$}\\
\hline
102001111000020102
100202100021010202
012010002021020011
001200020120021120\\
000101102022211000
110011020110110000
100000202010221202
000121220021002002\\
120112002000002120
101221012000002001
122020020200101001
010000210201011120\\
110122001100200001
101010201222000010
011100100200120220
012200020202202020\\
010022002002100112
001000120211200110\\
\hline
\multicolumn{1}{c}{$ A_{18,15,4}$}\\
\hline
012002001101110020
100020020012012201
012000210010201011
120210011000002101\\
110020110220200020
102102120000020110
000121010011122000
001020001200111110\\
122000202201010002
101102210102000002
010010002212100120
011012000201002210\\
101200102111001000
001100220001200121
010001002120012110
000111101010210002\\
100111000020101201
110201221000020002\\
\noalign{\hrule height1pt}
\end{tabular}
}
\end{table}

\begin{table}[thbp]
\caption{Mutually unbiased weighing matrices of order $19$ and weight $9$}
\label{Tab:19}
\centering
\medskip
{\footnotesize
\begin{tabular}{l}
\noalign{\hrule height1pt}
\multicolumn{1}{c}{$W_{19}$}\\
\hline
0100200020021001212
1010010100222000022
0102002002210012200
1002000022001001121\\
0101100200020100221
0121020000000210122
0100110111011001000
0010022210200221000\\
1000200201012111000
0122011200002220000
1002121010020010010
0012000201101002022\\
0110102020102000110
1001010200201002110
1011001002110200200
1020002012100120002\\
1220102021000200200
1100220101000022001
0000121020210120002\\
\hline
\multicolumn{1}{c}{$ A_{19,2}$}\\
\hline
1022102100001110000
1100101002200202002
0120000210210001201
0101010021010100012\\
1010100011000020111
1201020000001021202
0000001111110002220
1110000002102101020\\
1000200201220102020
1000200200111210100
1020212100002220000
0001110001020211020\\
1200001020002010211
0112022021000200200
0120001020121020001
0101220110020010010\\
0012211100201001000
0001020120210000121
0011012002001002201\\
\hline
\multicolumn{1}{c}{$ A_{19,3}$}\\
\hline
1001100110021000102
0121202011110000000
1000200200022201120
1110000001202012010\\
0001020120122100001
1220011001000022001
0100020200201120101
0000001110212101020\\
0011100021010221000
0010200112000220011
0110011000101010021
0100112002002022020\\
1212002000110100100
0122010120000001110
1102021000100020202
0001201022010002102\\
1001010202000101210
1000202120201000220
1020120002010210001\\
\hline
\multicolumn{1}{c}{$ A_{19,4}$}\\
\hline
1020001201200000212
0111002201002000022
0011101100201000220
0100201102002201002\\
1202100000000210122
1012020010012000201
0000001111022102100
0001110002012012010\\
1201200121110000000
0012210020020012200
1201202012220000000
1000110020022021001\\
0120002101200210001
0001021200120210001
0000122100120000212
1110000000001111110\\
0010020020200222110
1100010010101222000
1120020022000102020\\
\hline
\multicolumn{1}{c}{$ A_{19,5}$}\\
\hline
1022100211110000000
0000121022102010002
0120210000002100122
0121020000001022012\\
0100020200202202021
0112020100011100020
0100111220201001000
0100110110020012200\\
1000200202121000220
0011102212000100100
1011001001002120200
1010012020010202002\\
1000022001220011002
1220100102200020020
1101000100100201101
0012001010020220102\\
1000201012210010010
0102002002002021210
1002000020020102111\\
\hline
\multicolumn{1}{c}{$ A_{19,6}$}\\
\hline
1000110021010010101
0010202000200011211
1011001002200021100
1120020020022200010\\
1001102000200102202
0121201200001110000
1002001100020100221
0102000110201002110\\
0010020021101120010
0100112002121001000
0010020022001212020
0111210100102002000\\
0110101211000200200
1200000212100002011
0001022011020000121
1200210001021200002\\
1000020110110011002
0012000200022110102
1102202200010020020\\
\noalign{\hrule height1pt}
\end{tabular}
}
\end{table}

\begin{table}[thbp]
\caption{Mutually unbiased weighing matrices of order $21$ and weight $9$}
\label{Tab:21}
\centering
\medskip
{\footnotesize
\begin{tabular}{l}
\noalign{\hrule height1pt}
\multicolumn{1}{c}{$W_{21}$}\\
\hline
110100100210010010002
000011002010002122100
012002122020100020000\\
100010101100000001111
001200110100101002002
010001020101201100200\\
001110012021000001020
000102200000021100112
120002002000211002001\\
100001220000101200120
120000001201100120200
100010000002220220202\\
100000202100102010210
001102021100002002020
001121000022010020010\\
011000000201020202011
012010211020010002000
110020010002020100021\\
011202200010010021000
100220000021202000102
001210020222000110000\\
\hline
\multicolumn{1}{c}{$ A_{21,2}$}\\
\hline
000001021200002220210
001020200002120021100
102020100022001000011\\
100001010020012120020
100000022002202011002
000120011100022010200\\
000100221120011000002
000100022001100101201
001101000200000112110\\
001122100010210020000
100111010010001201000
101200200101200000011\\
010010100100020120012
010001020112000002021
011000000220221000220\\
112002201210000100000
110020002001100202002
010112000020002200101\\
011200101000110011000
120000121001020000120
121012000002100002200\\
\hline
\multicolumn{1}{c}{$ A_{21,3}$}\\
\hline
111000202100002000120
000011010102000212200
001100200201001210002\\
010001110011100000110
010002000212202200010
102002000121010200010\\
120001000201202002001
121000100002001220100
100000210222100100010\\
010120002000001022201
010200100220010002022
000010221010110022000\\
110200021000021010001
001201022000010001210
010111001020200021000\\
100100120000122000202
100020011010010001220
102010002010201100002\\
001020001100200102012
001212010001020020200
001112100000010110001\\
\noalign{\hrule height1pt}
\end{tabular}
}
\end{table}

\begin{landscape}

\begin{table}[thbp]
\caption{Mutually unbiased weighing matrices of order $22$ and weight $9$}
\label{Tab:22}
\centering
\medskip
{\footnotesize
\begin{tabular}{l}
\noalign{\hrule height1pt}
\multicolumn{1}{c}{$W_{22}$}\\
\hline
0010000212201000022001
0001000201022022120000
1020102000000012002120
0000011002000102110110
1002000012122020200000\\
1110000111200100000002
0001100010012201000110
1020002000000101100211
1200221000000001100120
0001200000222010210001\\
0121021002000000002202
0110202000100200112000
0001110000021021010020
1011000220110100200000
0102121201000000010001\\
0120210210010000001020
0001020110101002001001
0010100002002010101220
1100011120000200020001
0100022022200020001100\\
1200000000210222010200
1000000200021210001012\\
\hline
\multicolumn{1}{c}{$ A_{22,2}$}\\
\hline
1000010220000000011222
0100210012200000102002
1212000002020100100001
1000221000002220020200
0000101001210122100000\\
0011000010002121210000
1101100100021000000210
1002100010110200012000
1100000220200000202101
0000110102002002220020\\
0011010001002210100001
1020000011000111020020
0100001210120002001100
1010002000010000021112
0012211101001000200000\\
0010022001020002002022
0000001020102110002012
0011021002011010000020
1221200100000002010100
0100202000110102000201\\
0100000120100021100120
0102020100202010011000\\
\hline
\multicolumn{1}{c}{$ A_{22,3}$}\\
\hline
0110100011100002100200
1001020000002110010022
0011001200210000010011
0121200010021200010000
0010210002101110002000\\
0000012221220012000000
0000011101200001102020
1002202101010000010010
0100120120001010000101
1012020200020201002000\\
1120110002002000002010
0010000102220000101012
1200001100020002200201
0100220000200122022000
0000001011000010220112\\
1110010000000020201120
0100200000002011021201
1001000020011200020202
1201002010000000120101
0000102010201101200200\\
1022001200001100101000
0001000021120121000010\\
\hline
\multicolumn{1}{c}{$ A_{22,4}$}\\
\hline
0011210000011000021001
1000001210000221011000
0012101000020000220101
1210001100010100012000
0111000202002011002000\\
0120001200010102200010
0000001012021012100200
0000120002210020020220
1002200020200210200200
0100001101100001020202\\
0102000102010200100110
0011100001200202000012
1021100100002010001001
0000010000102222002201
1100010020220120100000\\
1200020220100000120010
1100022011001000002001
0012002010002100001210
0001220102020020200010
0110020020100002011020\\
1000200010002002020122
1000112002101000200002\\
\hline
\multicolumn{1}{c}{$ A_{22,5}$}\\
\hline
1000202002002022100010
1022020000201200010100
0100100120100202002002
1220010000022010200002
1001102001200000100220\\
1000001220120020000021
0101200002200200220020
1100001102000111100000
0012022020020010021000
0120000001000120021102\\
0000121200202000022010
0000011020201002001210
0100012200021001002010
0012010000200102002120
0100020010020102210200\\
0100200221012010010000
1011000101000000200111
0102110010002200001001
1012000000010021200202
1000000210111012020000\\
0010201011020200100002
0011100202000000011102\\
\noalign{\hrule height1pt}
\end{tabular}
}
\end{table}

\begin{table}[thbp]
\caption{Mutually unbiased weighing matrices of order $22$ and weight $9$}
\label{Tab:22-2}
\centering
\medskip
{\footnotesize
\begin{tabular}{l}
\noalign{\hrule height1pt}
\multicolumn{1}{c}{$ A_{22,6}$}\\
\hline
0110221100000000021100
1000020022002121002000
0010010100010221010001
0102000010200100100221
0001001001102102010001\\
1202000111020000002100
0100100021200000010122
1002010020002010101010
1001001000020211000220
0010101200021020100010\\
1102200200101000210000
0012120000002202200200
0000011212202000200100
0010100010110111000002
0120100001000001220011\\
1001022010200000011010
0120000010102220100002
1111012000000002022000
1020101102011002000000
0000102002120000001121\\
0000010100020120201202
1200000201010020021020\\
\hline
\multicolumn{1}{c}{$ A_{22,7}$}\\
\hline
0010001001200012010220
1020002001021100010002
0000100000000211202212
1110200000201220000010
1201120020210000000100\\
1020011010000210100100
0001110002000022100202
0000000110011000221022
1101001000022100220000
1012020100102000100002\\
0010012000200111120000
1000000022101010001201
0122000122200000002020
0100100110010100010011
0001210121110000002000\\
0101002000002201011020
0110100200101000002120
1002012200012002200000
0001222012000012002000
0120020201010000120200\\
1200000010000021002221
0000102101020202020001\\
\hline
\multicolumn{1}{c}{$ A_{22,8}$}\\
\hline
0100000000002222220110
1101200000100010000121
0011000001200021101100
0000101101101020002001
1212011000002200000020\\
0112200102000120010000
1000002110001202100002
1100110000010100021200
1000021022221000020000
1000122200000020210020\\
1200012020000100002110
0102100010200011002100
0011001210000102002002
0101010020200200012200
0001100102122001000002\\
0010000202111201000010
1020021000012000110010
1000200011020001200210
0010100000020012011011
0010020121010010200002\\
0102000221120000100002
0010022000002000122201\\
\hline
\multicolumn{1}{c}{$ A_{22,9}$}\\
\hline
1020200020102000200011
1010002010100102010100
0102022100010101020000
0102000100001202010201
0010110001110200220000\\
0001012120200000000121
1000100001200021012010
0000121020000102202020
0000001010012011010021
0110210200200100200200\\
1000002002002200002222
1202000000021001201020
1100100200020010120001
1001201110001000022000
1010001022010020101000\\
0102011100022000000102
0101000020101011010002
0010020002200210200110
0100220201000220000120
1020000001210012001002\\
0121100012000020201000
0011020101022000001200\\
\noalign{\hrule height1pt}
\end{tabular}
}
\end{table}

\begin{table}[thbp]
\caption{Mutually unbiased weighing matrices of order $23$ and weight $9$}
\label{Tab:23}
\centering
\medskip
{\footnotesize
\begin{tabular}{l}
\noalign{\hrule height1pt}
\multicolumn{1}{c}{$W_{23}$}\\
\hline
01000100010020201200101
01221012000010010000010
01102001002002010000012\\
10022100010010000002220
00100000011201000011210
12000010000022010110001\\
00001121101000010120000
00120110120100020001000
10000000201002222020010\\
01202200101120000000200
00012102020001212000000
01100022200100100100001\\
01001001220000200010220
10000001020001100202011
00100210000001201122000\\
00111002100002000202200
10001000012121002000002
10020022020220001000002\\
00120200001000012200120
10000220102010200001001
11010010000200100021020\\
01000000100200022112100
10010000001110001010102\\
\hline
\multicolumn{1}{c}{$ A_{23,2}$}\\
\hline
00110001200011020210000
00000110012010212000002
10001020200010010020011\\
00000102211000100001220
11010022100100002010000
01000012000000020221110\\
10000010020001200101021
10210211011002000000000
01002200202200102100000\\
01002101001021010000010
10121000010220000010100
10110000002020001020202\\
00000012200100001112010
12202002000201000202000
10122000021012000000002\\
00102000010100000022121
10220101002100120000000
01000100000202220002201\\
01200020200000201000122
01001010020000110202020
01020200110011001000200\\
00010100100210101100100
00001000001001022122002\\
\hline
\noalign{\hrule height1pt}
\end{tabular}
}
\end{table}

\begin{table}[thbp]
\caption{Mutually unbiased weighing matrices of order $24$ and weight $9$}
\label{Tab:24}
\centering
\medskip
{\scriptsize
\begin{tabular}{l}
\noalign{\hrule height1pt}
\multicolumn{1}{c}{$W_{24}$}\\
\hline
010201001200020000201002
010102011002000011000020
011101200000100002000101
101002202000002200001002
010020102002002100020010\\
000010010111000110001010
001000220012201100000200
112100020001020000100200
100000100010110002200220
100021000000011211000010\\
010210020000012001010001
000010102002021200001001
001000010201000021020201
100002001020200002200011
001000120001200010020120\\
001200000120120010002200
001020101100000020111000
001110100200000000012012
010010000120011020020002
120011001002002000120000\\
100202000200101100100100
100000000110020021202100
100001012020200100010020
010200010010200202102000\\
\hline
\multicolumn{1}{c}{$ A_{24,2}$}\\
\hline
100000020022012010001010
001000010001010010102210
101200000000120020010012
000101010100022000020110
011000102100201100001000\\
001022000010000211200100
000120021110000002000202
010000121001100000100101
001001220201000100220000
001200100002002002022020\\
010001000000022011010220
001100001202200021100000
112002000000000101022002
100020000020021200020201
000001022100010221002000\\
110110002210000202000000
120100100000000100212001
120001000000001010100122
010020200020200002012100
001112200120100000000020\\
100210201110200000000001
010001011002111000200000
000010101021200200200002
100020010001012020001020\\
\hline
\multicolumn{1}{c}{$ A_{24,3}$}\\
\hline
100001000200000220102210
122200001210000000001020
101000010021000211000020
000010000000002001211212
010000022200022000020022\\
012200202000111001000000
100010122111000000000001
010122001011000021000000
010000200000212202001001
000100000012001210220200\\
100000221120000100020200
000001200110200000012022
100020012102020020001000
010200101100100202000002
111000000200001102210000\\
010011001002020001000101
001220120002210001000000
001202200010020010100010
100112000002010000100102
001001010010112100020000\\
012000110000200110100200
000001000001201000021112
102020000000002010202110
000121020000100010111000\\
\hline
\multicolumn{1}{c}{$ A_{24,4}$}\\
\hline
102010011000010222000000
000000100221001000012120
100000021000100010102012
001111100010000000101100
001200000111001002010200\\
102000000011202101010000
010022000210000002000111
110102002100001201000000
010001000201200210020200
001120201000002200010020\\
001012010021022000000010
100001202000020002200102
000001202020010000110011
010001100100102000202001
000110201000001110200001\\
100200010002020010100021
100120100020000102001200
010010020202020020010200
010000200001100120120020
121000020000200020022001\\
011000010002210100002002
000012022000012012000020
101200000200110001201000
010200021120200000001100\\
\hline
\multicolumn{1}{c}{$ A_{24,5}$}\\
\hline
000100010000002202012011
001120202001110001000000
102020022010002002001000
101021110000200000120000
000111000012010020000202\\
100010000000002101220011
011000020002000100110110
000110002100001012020100
001010000101202000011020
100000021101001220000010\\
100210200200010210100000
111000000002120200200020
010102000200201000001201
010001001001100112000200
012000000110000001102021\\
010202012100000000000212
102100100020000011010002
000001222020220000002200
010000000211200000202102
000010102201120020100000\\
000102201000022000120002
120000210010021100010000
012001210020000020001100
100002000020010122002020\\
\hline
\multicolumn{1}{c}{$ A_{24,6}$}\\
\hline
112102100000012000000001
000000101120020000012110
100200100010201100002020
101001000000012001110002
000000120200000011220012\\
120020000000022200020120
010000000012101210100100
011110001201020000000020
100002210200000002002012
000000010001011221200100\\
010000022101200202000002
101002020022001020001000
010020011100000100021002
001101112002000002200000
010020002000120021002200\\
000001021000110022022000
001000002001000100120111
000120220010000100210100
110201200020000010200001
001020001010200200000211\\
001212000110102000200000
120100000101101010000200
102011000010020020001010
000110200102200001022000\\
\noalign{\hrule height1pt}
\end{tabular}
}
\end{table}

\end{landscape}
\end{document}